\documentclass[10pt]{article}
\usepackage{amsfonts}
\usepackage{amsmath,amssymb}

\newtheorem{theorem}{Theorem}[section]

\newtheorem{definition}[theorem]{Definition}

\newtheorem{lemma}[theorem]{Lemma}
\newtheorem{proposition}[theorem]{Proposition}
\newtheorem{remark}[theorem]{Remark}
\newenvironment{proof}[1][Proof]{\noindent \textbf{#1.} }{\  $\Box$}
 \numberwithin{equation}{section}
\begin{document}

\title{Backward Stochastic Differential Equations Driven by $G$-Brownian
Motion}
\author{Mingshang Hu \thanks{%
School of Mathematics, Shandong University, humingshang@sdu.edu.cn} \and %
Shaolin Ji\thanks{%
Qilu Institute of Finance, Shandong University, jsl@sdu.edu.cn } \and Shige
Peng\thanks{%
School of Mathematics and Qilu Institute of Finance, Shandong University,
peng@sdu.edu.cn, Hu, Ji, and Peng's research was partially supported by NSF
of China No. 10921101; and by the 111 Project No. B12023} \and Yongsheng Song\thanks{%
Academy of Mathematics and Systems Science, CAS, Beijing, China,
yssong@amss.ac.cn. Research supported by by NCMIS; Youth Grant of
National Science Foundation (No. 11101406); Key Lab of Random
Complex Structures and Data Science, CAS (No. 2008DP173182).} }
\maketitle
\date{}

\begin{abstract}
In this paper, we study the following of backward stochastic differential
equations driven by a $G$-Brownian motion $(B_{t})_{t\geq 0}$ in the
following form:
\begin{eqnarray*}
Y_{t} &=&\xi
+\int_{t}^{T}f(s,Y_{s},Z_{s})ds+\int_{t}^{T}g(s,Y_{s},Z_{s})d\langle
B\rangle _{s} \\
&&-\int_{t}^{T}Z_{s}dB_{s}-(K_{T}-K_{t}).
\end{eqnarray*}%
Under a Lipschitz condition of $f$ and $g$ in $Y$ and $Z$. The existence and
uniqueness of the solution $(Y,Z,K)$ is proved, where $K$ is a decreasing $G$%
-martingale.
\end{abstract}


\textbf{MSC-classification}: 60H10, 60H30

\section{Introduction}

A typical classical Backward Stochastic Differential Equation, BSDE in
short, is defined on a Wiener probability space $(\Omega ,\mathcal{F},P)$ in
which $\Omega $ is the space of continuous paths. A standard Brownian motion
is defined as the canonical process, namely $B_{t}(\omega )=\omega _{t}$,
for $\omega \in \Omega $, together with its natural filtration $\mathbb{F}=(%
\mathcal{F}_{t})_{t\geq 0}$. The problem is to solve a pair of $\mathbb{F}$%
-adapted processes $(Y,Z)$ satisfying the following BSDE
\begin{equation}
Y_{t}=\xi +\int_{t}^{T}g(s,Y_{s},Z_{s})ds-\int_{t}^{T}Z_{s}dB_{s},
\label{fBSDE}
\end{equation}%
where $g$ is a given function, called the generator of (\ref{fBSDE}), and $%
\xi $ is a given $\mathcal{F}_{T}$-measurable random variable called the
terminal condition of the BSDE.

Linear BSDE was introduced by Bismut \cite[1973]{Bismut}. The basic
existence and uniqueness theorem of nonlinear BSDEs, with a
Lipschitz condition of $g$ with respect to $(y,z)$, was obtained in
Pardoux \& Peng \cite[1990]{PP90}. Peng \cite[1991a]{Peng1991}
established a probabilistic interpretation, through BSDE, of system
of quasi-linear partial differential equations, PDE in short, of
parabolic and elliptic types, under a strong elliptic assumption.
Then Peng \cite[1992]{Peng1992} and Pardoux \& Peng
\cite[1992]{PP92} obtained this interpretation for possibly
degenerate situation. This interpretation which established a 1-1
correspondence between a solution of a PDE and the corresponding
state dependent BSDE is the so-called nonlinear Feynman-Kac formula.
Since then and specially after the study of BSDE in \cite{EPQ} with
application to finance, BSDE theory
has been extensively studied. We refer to a survey paper of \cite%
{PengICM2010} for more details of the theoretical studies and applications
to, e.g., stochastic controls, optimizations, games and finance.

Under some suitable condition imposed to the generator $g$, this BSDE was
used to define a nonlinear expectation $\mathcal{E}^{g}[\xi ]:=Y_{0}$,
called $g$-expectation (see \cite[Peng1997]{Peng1997}). This $g$-expectation
is time consistent, namely the conditional expectation $\mathcal{E}^{g}[\xi |%
\mathcal{F}_{t}]$ is well-defined, under which the solution process $Y_{t}$
is a nonlinear martingale $Y_{t}=\mathcal{E}^{g}[\xi |\mathcal{F}_{t}]$. In
fact it was proved that there exists a 1-1 correspondence between a set of
`dominated' and time-consistent nonlinear expectations and that of BSDEs
(see \cite{CHMP}).

There are at least two reasons to study BSDEs and/or the corresponding
time-consistent nonlinear expectations outside of a classical probability
space framework. The first one is that the classical BSDE can provide a
probabilistic interpretation of a PDE only for quasilinear but not fully
nonlinear cases. The second one is that the well-known HJB-equation method
of volatility model uncertainty (see \cite{Avel1995}) is difficult to treat
a general path-dependent situation to measure financial risks. This problem
is also closely related to defining an important type of time-consistent
coherent risk measures (or sunlinear expectation) for which the
probabilities involved in the robust representation theorem are singular
from each others.

The notion of time-consistent fully nonlinear expectations has been
established in \cite[Peng2004]{Peng2004} and \cite[Peng2005]{Peng2005}. The
main approach of \cite{Peng2004} is to establish a new type of
`path-dependent value function' of a stochastic optimal control system, in
which the time consistency was able to be obtained through the corresponding
path-dependent dynamic programming principle (DPP).

In \cite{Peng2005} a canonical space of nonlinear Markovian paths was
defined. A nonlinear expectation, together with it's time-consistent
conditional expectations, was defined firstly on a subspace of
finite-dimensional cyclic functions of canonical paths, through a sublinear
Markovian semigroup, step by step and backwardly in time. This expectation
and the corresponding conditional expectations were then extended to the
completion of the above subspace of cyclic functions by using the Banach
norm induced by this sublinear expectation. Existence and uniqueness of a
type of multi-dimensional fully nonlinear BSDE was obtained in this paper.

As a typical and important situation of the above nonlinear Markovian
processes, Peng (2006) introduced a framework of time consistent nonlinear
expectation called $G$-expectation $\hat{\mathbb{E}}[\cdot ]$ (see lecture
notes of \cite{P10} and the references therein) in which a new type of
Brownian motion called $G$-Brownian motion was constructed and the
corresponding stochastic calculus of It\^{o}'s type was established.

Using this stochastic calculus the existence and uniqueness of SDEs driven
by $G$-Brownian motion can be obtained, in a way parallel to that of
classical theory of SDE, through which a large set of fully nonlinear
Markovian and non Markovian processes can be easily generated. But the
corresponding BSDE driven by a $G$-Brownian motion $(B_{t})_{t\geq 0}$
becomes a challenging and fascinating problem.

Just like in the classical situation, the first and most simplest BSDE in
this $G$-framework is the corresponding $G$-martingale representation
theorem. For a dense family of $G$-martingales, Peng \cite{P07b} obtained
the following result: a $G$-martingale $M$ is of the form
\begin{eqnarray*}
M_{t} &=&M_{0}+\bar{M}_{t}+K_{t},\ \  \\
\bar{M}_{t} &:&=\int_{0}^{t}z_{s}B_{s},\ \ \ \ \ K_{t}:=\int_{0}^{t}\eta
_{s}\left\langle B\right\rangle _{s}-\int_{0}^{t}2G(\eta _{s})ds.
\end{eqnarray*}%
Here $M$ is decomposed into two types of very different $G$-martingales: the
first one $\bar{M}$ is called symmetric $G$-martingale for which $-\bar{M}$
is also a $G$-martingale. The second one $K$ is quite unusual since it is a
decreasing process. How to understand this new type of decreasing $G$%
-martingales has become a main concern in the theory of $G$-framework, which
rised an interesting open problem (see \cite{P07b} and \cite{P10}).

An important step is to decompose an $G$-martingale $M$ into a sum
of a symmetric $G$-martingale $\bar{M}$ and a decreasing
$G$-martingale $K$. This difficult problem was solved after a series
of successive efforts of Soner, Touzi \& Zhang \cite[2011]{STZ} and
Song \cite[2011]{Song11}, \cite[2012]{Song12}. Another important
step is to give a completion of random variables in which
the non increasing $G$-martingales $K$ in the decomposition of the $G$%
-martingale $\mathbb{\hat{E}}_{t}[\xi ]$ can be uniquely represented $%
K_{t}:=\int_{0}^{t}\eta _{s}\left\langle B\right\rangle
_{s}-\int_{0}^{t}2G(\eta _{s})ds$. Thanks to an original new norm introduced
in Song \cite[2012]{Song12} for decreasing $G$-martingales, a representation
theorem of $G$-martingales in a complete subspace of $L_{G}^{\alpha }(\Omega
_{T})$ has been obtained by Peng, Song and Zhang \cite[2012]{PSZ2012}.

In considering the above $G$-martingale representation theorem, a natural
formulation of a BSDE driven by $G$-Brownian motion is to find a triple of
processes $(Y,Z,K)$, where $K$ is a decreasing $G$-martingale, satisfying
\begin{eqnarray}
Y_{t} &=&\xi
+\int_{t}^{T}f(s,Y_{s},Z_{s})ds+\int_{t}^{T}g(s,Y_{s},Z_{s})d\langle
B\rangle _{s}  \label{e1} \\
&&-\int_{t}^{T}Z_{s}dB_{s}-(K_{T}-K_{t}).  \notag
\end{eqnarray}%
The main result of this paper is the existence and uniqueness of a triple $%
(Y,Z,K)$ which solves BSDE (\ref{e1}), see Theorem \ref{the4.1} and \ref%
{the4.4}.

To prove the existence and uniqueness, two new approaches have been
introduced. The first one is applying the partition of unity theorem to
construct a new type of Galerkin approximation, in the place of the
well-Known Picard approximation approach frequently used in classical BSDE
theory. The second one involves Lemma 3.4 for decreasing $G$-martingales,
which helps us to use our $G$-stochastic calculus obtain the uniqueness, as
well as the existence part of the proof. Estimate (\ref{e2song}) originally
obtained in \cite{Song11} also plays an important role.

Now let us compare the results of this paper with the existing results
concerning fully nonlinear BSDEs.

For the case where the generator $f$ in (1.2) is independent of $z$ and $g=0$%
, the above problem can be equivalently formulated as
\begin{equation*}
Y_{t}=\hat{\mathbb{E}}_{t}[\xi +\int_{t}^{T}f(s,Y_{s})ds].
\end{equation*}%
The existence and uniqueness of such fully nonlinear BSDE was obtained in
\cite[Peng2005]{Peng2005} and \cite{P07b}, \cite{P10}. This approach was
used to treat many interesting problem corresponding fully nonlinear PDE
and/or system of fully nonlinear PDEs, in which each component $u^{i}$ of
the solution is associated to its own second order nonlinear elliptic
operator (see \cite{P10}). But a drawback of this formulation is that it is
difficult to treat the case where the generators $f$ and/or $g$ contain the $%
z$-terms (but the $z$-term can be integrated in the nonlinear Markovian
semigroup, see \cite[Peng2005]{Peng2005} and \cite{P10}).

Soner, Touzi and Zhang \cite[2012]{STZ11} have obtained an existence
and uniqueness theorem for a type of fully nonlinear BSDE, called
2BSDE, whose
generator can contain $Z$-term. Their solution is $(Y,Z,K^{\mathbb{P}})_{%
\mathbb{P}\in \mathcal{P}_{H}^{\kappa }}$ which solves, for each probability
$\mathbb{P}\in \mathcal{P}_{H}^{\kappa }$, the following BSDE
\begin{equation*}
Y_{t}=\xi +\int_{t}^{T}F_{s}(Y_{s},Z_{s})ds-\int_{t}^{T}Z_{s}dB_{s}+(K_{T}^{%
\mathbb{P}}-K_{t}^{\mathbb{P}}),\ \ \mathbb{P}\text{-a.s., }
\end{equation*}%
for which the following minimum condition is satisfied
\begin{equation*}
K_{t}^{\mathbb{P}}=\text{ess}\inf_{\mathbb{P}^{\prime }\in \mathcal{P}%
_{H}^{\kappa }(t+,\mathbb{P})}\mathbb{E}_{t}^{\mathbb{P}^{\prime }}[K_{T}^{%
\mathbb{P}}],\ \ \mathbb{P}\text{-a.s.,\ \ }\forall \mathbb{P}\in \mathcal{P}%
_{H}^{\kappa },\ t\in \lbrack 0,T].
\end{equation*}
But in their paper the processes $(K^{\mathbb{P}})_{\mathbb{P}\in \mathcal{P}%
_{H}^{\kappa }}$ are not able to be \textquotedblleft
aggregated\textquotedblright\ into an `universal $K$'. This is a drawback in
the sense that the quantity of calculation for solving this 2BSDE is still
involved an complicated optimization problem with respect to the original
subset of probabilities $\mathcal{P}_{H}^{\kappa }$. In our paper the triple
$(Y,Z,K)$ is universally defined within the $G$-Brownian motion framework.
The method of our paper can be also applied to many other situations.

The paper is organized as follows. In section 2, we present some
preliminaries for stochastic calculus under $G$-framework. Some estimates
for the solution of $G$-BSDE are established in section 3. In section 4 the
existence and uniqueness theory is provided.

\section{Preliminaries}

We review some basic notions and results of $G$-expectation and the related
space of random variables. More details of this section can be found in \cite%
{P07a}, \cite{P07b}, \cite{P08a}, \cite{P08b}, \cite{P10}.

\begin{definition}
\label{def2.1} Let $\Omega$ be a given set and let $\mathcal{H}$ be a vector
lattice of real valued functions defined on $\Omega$, namely $c\in \mathcal{H%
}$ for each constant $c$ and $|X|\in \mathcal{H}$ if $X\in \mathcal{H}$. $%
\mathcal{H}$ is considered as the space of random variables. A sublinear
expectation $\mathbb{\hat{E}}$ on $\mathcal{H}$ is a functional $\mathbb{%
\hat {E}}:\mathcal{H}\rightarrow \mathbb{R}$ satisfying the following
properties: for all $X,Y\in \mathcal{H}$, we have

\begin{description}
\item[(a)] Monotonicity: If $X\geq Y$ then $\mathbb{\hat{E}}[X]\geq \mathbb{%
\hat{E}}[Y]$;

\item[(b)] Constant preservation: $\mathbb{\hat{E}}[c]=c$;

\item[(c)] Sub-additivity: $\mathbb{\hat{E}}[X+Y]\leq \mathbb{\hat{E}}[X]+%
\mathbb{\hat{E}}[Y]$;

\item[(d)] Positive homogeneity: $\mathbb{\hat{E}}[\lambda X]=\lambda
\mathbb{\hat{E}}[X]$ for each $\lambda \geq0$.
\end{description}

$(\Omega,\mathcal{H},\mathbb{\hat{E}})$ is called a sublinear expectation
space.
\end{definition}

\begin{definition}
\label{def2.2} Let $X_{1}$ and $X_{2}$ be two $n$-dimensional random vectors
defined respectively in sublinear expectation spaces $(\Omega_{1},\mathcal{H}%
_{1},\mathbb{\hat{E}}_{1})$ and $(\Omega_{2},\mathcal{H}_{2},\mathbb{\hat{E}}%
_{2})$. They are called identically distributed, denoted by $X_{1}\overset{d}%
{=}X_{2}$, if $\mathbb{\hat{E}}_{1}[\varphi(X_{1})]=\mathbb{\hat{E}}%
_{2}[\varphi(X_{2})]$, for all$\ \varphi \in C_{l.Lip}(\mathbb{R}^{n})$,
where $C_{l.Lip}(\mathbb{R}^{n})$ is the space of real continuous functions
defined on $\mathbb{R}^{n}$ such that
\begin{equation*}
|\varphi(x)-\varphi(y)|\leq C(1+|x|^{k}+|y|^{k})|x-y|\ \text{\ for all}\
x,y\in \mathbb{R}^{n},
\end{equation*}
where $k$ and $C$ depend only on $\varphi$.
\end{definition}

\begin{definition}
\label{def2.3} In a sublinear expectation space $(\Omega,\mathcal{H},\mathbb{%
\hat{E}})$, a random vector $Y=(Y_{1},\cdot \cdot \cdot,Y_{n})$, $Y_{i}\in
\mathcal{H}$, is said to be independent of another random vector $%
X=(X_{1},\cdot \cdot \cdot,X_{m})$, $X_{i}\in \mathcal{H}$ under $\mathbb{%
\hat {E}}[\cdot]$, denoted by $Y\bot X$, if for every test function $\varphi
\in C_{l.Lip}(\mathbb{R}^{m}\times \mathbb{R}^{n})$ we have $\mathbb{\hat{E}}%
[\varphi(X,Y)]=\mathbb{\hat{E}}[\mathbb{\hat{E}}[\varphi(x,Y)]_{x=X}]$.
\end{definition}

\begin{definition}
\label{def2.4} ($G$-normal distribution) A $d$-dimensional random vector $%
X=(X_{1},\cdot \cdot \cdot,X_{d})$ in a sublinear expectation space $(\Omega,%
\mathcal{H},\mathbb{\hat{E}})$ is called $G$-normally distributed if for
each $a,b\geq0$ we have
\begin{equation*}
aX+b\bar{X}\overset{d}{=}\sqrt{a^{2}+b^{2}}X,
\end{equation*}
where $\bar{X}$ is an independent copy of $X$, i.e., $\bar{X}\overset{d}{=}X$
and $\bar{X}\bot X$. Here the letter $G$ denotes the function
\begin{equation*}
G(A):=\frac{1}{2}\mathbb{\hat{E}}[\langle AX,X\rangle]:\mathbb{S}%
_{d}\rightarrow \mathbb{R},
\end{equation*}
where $\mathbb{S}_{d}$ denotes the collection of $d\times d$ symmetric
matrices.
\end{definition}

Peng \cite{P08b} showed that $X=(X_{1},\cdot \cdot \cdot,X_{d})$ is $G$%
-normally distributed if and only if for each $\varphi \in C_{l.Lip}(\mathbb{%
R}^{d})$, $u(t,x):=\mathbb{\hat{E}}[\varphi(x+\sqrt{t}X)]$, $(t,x)\in
\lbrack 0,\infty)\times \mathbb{R}^{d}$, is the solution of the following $G$%
-heat equation:%
\begin{equation*}
\partial_{t}u-G(D_{x}^{2}u)=0,\ u(0,x)=\varphi(x).
\end{equation*}

The function $G(\cdot):\mathbb{S}_{d}\rightarrow \mathbb{R}$ is a monotonic,
sublinear mapping on $\mathbb{S}_{d}$ and $G(A)=\frac{1}{2}\mathbb{\hat{E}}%
[(AX,X)]\leq \frac{1}{2}|A|\mathbb{\hat{E}}[|X|^{2}]=:\frac{1}{2}|A|\bar{%
\sigma}^{2}$ implies that there exists a bounded, convex and closed subset $%
\Gamma \subset \mathbb{S}_{d}^{+}$ such that
\begin{equation*}
G(A)=\frac{1}{2}\sup_{\gamma \in \Gamma}\mathrm{tr}[\gamma A],
\end{equation*}
where $\mathbb{S}_{d}^{+}$ denotes the collection of nonnegative elements in
$\mathbb{S}_{d}$.

In this paper we only consider non-degenerate $G$-normal distribution, i.e.,
there exists some $\underline{\sigma}^{2}>0$ such that $G(A)-G(B)\geq
\underline{\sigma}^{2}\mathrm{tr}[A-B]$ for any $A\geq B$.

\begin{definition}
\label{def2.5} i) Let $\Omega_{T}=C_{0}([0,T];\mathbb{R}^{d})$, the
space of real valued continuous functions on $[0,T]$ with
$\omega_{0}=0$, be endowed with the supremum norm and let
$B_{t}(\omega)=\omega_{t}$ be the canonical process. Set
$\mathcal{H}_{T}^{0}:=\{
\varphi(B_{t_{1}},...,B_{t_{n}}):n\geq1,t_{1},...,t_{n}\in
\lbrack0,T],\varphi \in C_{l.Lip}(\mathbb{R}^{d\times n})\}$. $G$%
-expectation is a sublinear expectation defined by
\begin{equation*}
\mathbb{\hat{E}}[X]=\mathbb{\tilde{E}}[\varphi(\sqrt{t_{1}-t_{0}}%
\xi_{1},\cdot \cdot \cdot,\sqrt{t_{m}-t_{m-1}}\xi_{m})],
\end{equation*}
for all $X=\varphi(B_{t_{1}}-B_{t_{0}},B_{t_{2}}-B_{t_{1}},\cdot \cdot
\cdot,B_{t_{m}}-B_{t_{m-1}})$, where $\xi_{1},\cdot \cdot \cdot,\xi_{n}$ are
identically distributed $d$-dimensional $G$-normally distributed random
vectors in a sublinear expectation space $(\tilde{\Omega},\tilde{\mathcal{H}}%
,\mathbb{\tilde{E}})$ such that $\xi_{i+1}$ is independent of $%
(\xi_{1},\cdot \cdot \cdot,\xi_{i})$ for every $i=1,\cdot \cdot \cdot,m-1$. $%
(\Omega _{T},\mathcal{H}_{T}^{0},\mathbb{\hat{E}})$ is called a $G$%
-expectation space.

ii) Let us define the conditional $G$-expectation $\mathbb{\hat{E}}_{t}$ of $%
\xi \in \mathcal{H}_{T}^{0}$ knowing $\mathcal{H}_{t}^{0}$, for $t\in
\lbrack0,T]$. Without loss of generality we can assume that $\xi$ has the
representation $\xi=\varphi(B_{t_{1}}-B_{t_{0}},B_{t_{2}}-B_{t_{1}},\cdot
\cdot \cdot,B_{t_{m}}-B_{t_{m-1}})$ with $t=t_{i}$, for some $1\leq i\leq m$%
, and we put
\begin{equation*}
\mathbb{\hat{E}}_{t_{i}}[\varphi(B_{t_{1}}-B_{t_{0}},B_{t_{2}}-B_{t_{1}},%
\cdot \cdot \cdot,B_{t_{m}}-B_{t_{m-1}})]
\end{equation*}%
\begin{equation*}
=\tilde{\varphi}(B_{t_{1}}-B_{t_{0}},B_{t_{2}}-B_{t_{1}},\cdot \cdot
\cdot,B_{t_{i}}-B_{t_{i-1}}),
\end{equation*}
where
\begin{equation*}
\tilde{\varphi}(x_{1},\cdot \cdot \cdot,x_{i})=\mathbb{\hat{E}}[\varphi
(x_{1},\cdot \cdot \cdot,x_{i},B_{t_{i+1}}-B_{t_{i}},\cdot \cdot
\cdot,B_{t_{m}}-B_{t_{m-1}})].
\end{equation*}
\end{definition}

Define $\Vert \xi \Vert_{p,G}=(\mathbb{\hat{E}}[|\xi|^{p}])^{1/p}$ for $\xi
\in \mathcal{H}_{T}^{0}$ and $p\geq1$. Then \textmd{for all}$\ t\in \lbrack
0,T]$, $\mathbb{\hat{E}}_{t}[\cdot]$ is a continuous mapping on $\mathcal{H}%
_{T}^{0}$ w.r.t. the norm $\Vert \cdot \Vert_{1,G}$. Therefore it can be
extended continuously to the completion $L_{G}^{1}(\Omega_{T})$ of $\mathcal{%
H}_{T}^{0}$ under the norm $\Vert \cdot \Vert_{1,G}$.

Let $L_{ip}(\Omega_{T}):=\{ \varphi(B_{t_{1}},...,B_{t_{n}}):n\geq
1,t_{1},...,t_{n}\in \lbrack0,T],\varphi \in C_{b.Lip}(\mathbb{R}^{d\times
n})\},$ where $C_{b.Lip}(\mathbb{R}^{d\times n})$ denotes the set of bounded
Lipschitz functions on $\mathbb{R}^{d\times n}$. Denis et al. \cite{DHP11}
proved that the completions of $C_{b}(\Omega_{T})$ (the set of bounded
continuous function on $\Omega_{T}$), $\mathcal{H}_{T}^{0}$ and $%
L_{ip}(\Omega_{T})$ under $\Vert \cdot \Vert_{p,G}$ are the same and we
denote them by $L_{G}^{p}(\Omega_{T})$.

\begin{definition}
\label{def2.6} Let $M_{G}^{0}(0,T)$ be the collection of processes in the
following form: for a given partition $\{t_{0},\cdot \cdot \cdot,t_{N}\}=\pi
_{T}$ of $[0,T]$,
\begin{equation*}
\eta_{t}(\omega)=\sum_{j=0}^{N-1}\xi_{j}(\omega)I_{[t_{j},t_{j+1})}(t),
\end{equation*}
where $\xi_{i}\in L_{ip}(\Omega_{t_{i}})$, $i=0,1,2,\cdot \cdot \cdot,N-1$.
For $p\geq1$ and $\eta \in M_{G}^{0}(0,T)$, let $\Vert \eta
\Vert_{H_{G}^{p}}=\{ \mathbb{\hat{E}}[(\int_{0}^{T}|\eta_{s}|^{2}ds)^{p/2}]%
\}^{1/p}$, $\Vert \eta \Vert_{M_{G}^{p}}=\{ \mathbb{\hat{E}}%
[\int_{0}^{T}|\eta_{s}|^{p}ds]\}^{1/p}$ and denote by $H_{G}^{p}(0,T)$, $%
M_{G}^{p}(0,T)$ the completions of $M_{G}^{0}(0,T)$ under the norms $\Vert
\cdot \Vert_{H_{G}^{p}}$, $\Vert \cdot \Vert_{M_{G}^{p}}$ respectively.
\end{definition}

\begin{theorem}
\label{the2.7} (\cite{DHP11,HP09}) There exists a tight subset $\mathcal{P}%
\subset \mathcal{M}_{1}(\Omega_{T})$, the set of probability measures on $%
(\Omega_{T},\mathcal{B}(\Omega_{T}))$, such that
\begin{equation*}
\mathbb{\hat{E}}[\xi]=\sup_{P\in \mathcal{P}}E_{P}[\xi]\ \ \text{for \ all}\
\xi \in \mathcal{H}_{T}^{0}.
\end{equation*}
$\mathcal{P}$ is called a set that represents $\mathbb{\hat{E}}$.
\end{theorem}

\begin{remark}
\label{rem2.8} Denis et al. \cite{DHP11} gave a concrete set $\mathcal{P}%
_{M} $ that represents $\mathbb{\hat{E}}$. For simplicity, we only introduce
the $1$-dimensional case, i.e., $\Omega_{T}=C_{0}([0,T];\mathbb{R})$ .

Let $(\Omega^{0},\mathcal{F}^{0},P^{0})$ be a probability space and $%
\{W_{t}\}$ be a $1$-dimensional Brownian motion under $P^{0}$. Let $F^{0}=\{
\mathcal{F}_{t}^{0}\}$ be the augmented filtration generated by $W$. Denis
et al. \cite{DHP11} proved that
\begin{equation*}
\mathcal{P}_{M}:=\{P_{h}:P_{h}=P^{0}\circ
X^{-1},X_{t}=\int_{0}^{t}h_{s}dW_{s},h\in L_{F^{0}}^{2}([0,T];[\underline{%
\sigma},\overline{\sigma}])\}
\end{equation*}
is a set that represents $\mathbb{\hat{E}}$, where $L_{F^{0}}^{2}([0,T];[%
\underline{\sigma},\overline{\sigma}])$ is the collection of $F^{0}$-adapted
measurable processes with $\underline{\sigma}\leq|h_{s}|\leq \overline{\sigma%
}$. Here
\begin{equation*}
\underline{\sigma}^{2}:=-\mathbb{\hat{E}}[-B_{1}^{2}]\leq \mathbb{\hat{E}}%
[B_{1}^{2}]=:\overline{\sigma}^{2}.
\end{equation*}
For this 1-dimensional case,
\begin{equation*}
G(a)=\frac{1}{2}\mathbb{\hat{E}}[aB_{1}^{2}]=\frac{1}{2}[\overline{\sigma}%
^{2}a^{+}-\underline{\sigma}^{2}a^{-}].
\end{equation*}
\end{remark}

Let $\mathcal{P}$ be a weakly compact set that represents $\mathbb{\hat{E}}$%
. For this $\mathcal{P}$, we define capacity%
\begin{equation*}
c(A):=\sup_{P\in \mathcal{P}}P(A),\ A\in \mathcal{B}(\Omega_{T}).
\end{equation*}
A set $A\subset \Omega_{T}$ is polar if $c(A)=0$. A property holds
\textquotedblleft quasi-surely\textquotedblright \ (q.s. for short) if it
holds outside a polar set. In the following, we do not distinguish two
random variables $X$ and $Y$ if $X=Y$ q.s.. We set%
\begin{equation*}
\mathbb{L}^{p}(\Omega_{t}):=\{X\in \mathcal{B}(\Omega_{t}):\sup_{P\in
\mathcal{P}}E_{P}[|X|^{p}]<\infty \} \ \text{for}\ p\geq1.
\end{equation*}
It is important to note that $L_{G}^{p}(\Omega_{t})\subset \mathbb{L}%
^{p}(\Omega_{t})$. We extend $G$-expectation $\mathbb{\hat{E}}$ to $\mathbb{L%
}^{p}(\Omega_{t})$ and still denote it by $\mathbb{\hat{E}}$, for each $X\in$
$\mathbb{L}^{1}(\Omega_{T})$, we set%
\begin{equation*}
\mathbb{\hat{E}}[X]=\sup_{P\in \mathcal{P}}E_{P}[X].
\end{equation*}
For $p\geq1$, $\mathbb{L}^{p}(\Omega_{t})$ is a Banach space under the norm $%
(\mathbb{\hat{E}}[|\cdot|^{p}])^{1/p}$.

Furthermore, we extend the definition of conditional $G$-expectation. For
each fixed $t\geq0,$ let $(A_{i})_{i=1}^{n}$ be a partition of $\mathcal{B}%
(\Omega_{t})$, and set
\begin{equation*}
\xi=\sum_{i=1}^{n}\eta_{i}I_{A_{i}},
\end{equation*}
where $\eta_{i}\in L_{G}^{1}(\Omega_T)$, $i=1,\cdots,n$. We define the
corresponding conditional $G$-expectation, still denoted by $\mathbb{\hat{E}}%
_{s}[\cdot]$, by setting%
\begin{equation*}
\mathbb{\hat{E}}_{s}[\sum_{i=1}^{n}\eta_{i}I_{A_{i}}]:=\sum_{i=1}^{n}\mathbb{%
\hat{E}}_{s}[\eta_{i}]I_{A_{i}}\ \text{\ for}\ s\geq t.
\end{equation*}
The following lemma shows that the above definition of conditional $G$%
-expectation is meaningful.

\begin{lemma}
\label{lemA.7} For each $\xi,\eta \in L_{G}^{1}(\Omega_T)$ and $A\in
\mathcal{B}(\Omega_{t})$, if $\xi I_{A}\geq\eta I_{A}$ q.s., then $\mathbb{%
\hat {E}}_{t}[\xi]I_{A}\geq\mathbb{\hat{E}}_{t}[\eta]I_{A}$ q.s..
\end{lemma}

\begin{proof}
Otherwise, we can choose a compact set $K\subset A$ with $c(K)>0$ such that $%
(\mathbb{\hat {E}}_{t}[\xi]-\mathbb{\hat{E}}_{t}[\eta])^->0$ on $K$. Since $%
K $ is compact, we can choose a sequence of nonnegative functions $\{
\zeta_{n}\}_{n=1}^{\infty }\subset C_{b}(\Omega_{t})$ such that $%
\zeta_{n}\downarrow I_{K}$. By Theorem 31 in \cite{DHP11}, we have%
\begin{equation*}
\mathbb{\hat{E}}[\zeta_{n}(\xi-\eta)^-]\downarrow \mathbb{\hat{E}}[I_{K}(\xi
-\eta)^-]
\end{equation*}
and%
\begin{equation*}
\mathbb{\hat{E}}[\zeta_{n}\mathbb{\hat{E}}_{t}[(\xi-\eta)^-]]\downarrow
\mathbb{\hat{E}}[I_{K}\mathbb{\hat{E}}_{t}[(\xi-\eta)^-]].
\end{equation*}
Since
\begin{equation*}
\mathbb{\hat{E}}[\zeta_{n}(\xi-\eta)^-]=\mathbb{\hat{E}}[\zeta_{n}\mathbb{%
\hat{E}}_{t}[(\xi-\eta)^-]],
\end{equation*}
we have
\begin{equation*}
\mathbb{\hat{E}}[I_{K}\mathbb{\hat{E}}_{t}[(\xi-\eta)^-]]=\mathbb{\hat{E}}%
[I_{K}(\xi-\eta)^-]=0.
\end{equation*}
Noting that
\begin{equation*}
(\mathbb{\hat{E}}_{t}[\xi]-\mathbb{\hat{E}}_{t}[\eta])^-\leq \mathbb{\hat{E}}%
_{t}[(\xi-\eta)^-],
\end{equation*}
we get $\mathbb{\hat{E}}_{t}[(\xi-\eta)^-]>0$ on $K$. Also by $c(K)>0$ we
get $\mathbb{\hat{E}}[I_{K}\mathbb{\hat{E}}_{t}[(\xi-\eta)^-]>0$. This is a
contradiction and the proof is complete.
\end{proof}

We set%
\begin{equation*}
\mathbb{L}_{G}^{0,p,t}(\Omega_T):=\{
\xi=\sum_{i=1}^{n}\eta_{i}I_{A_{i}}:A_{i}\in \mathcal{B}(\Omega_{t}),%
\eta_{i}\in L_{G}^{p}(\Omega),n\in \mathbb{N}\}.
\end{equation*}
We have the following properties.

\begin{proposition}
\label{proA.8} For each $\xi,\eta \in \mathbb{L}_{G}^{0,1,t}(\Omega_T)$, we
have

\begin{description}
\item[(i)] Monotonicity: If $\xi \leq \eta$, then $\mathbb{\hat{E}}_{s}[\xi
]\leq \mathbb{\hat{E}}_{s}[\eta]$ for any $s\geq t$;

\item[(ii)] Constant preserving: If $\xi \in \mathbb{L}_{G}^{1,t}(%
\Omega_{t}) $, then $\mathbb{\hat{E}}_{t}[\xi]=\xi$;

\item[(iii)] Sub-additivity: $\mathbb{\hat{E}}_{t}[\xi+\eta]\leq \mathbb{%
\hat{E}}_{t}[\xi]+\mathbb{\hat{E}}_{t}[\eta]$;

\item[(iv)] Positive homogeneity: If $\xi \in \mathbb{L}_{G}^{0,\infty,t}(%
\Omega_{t})$ and $\xi\geq0$, then $\mathbb{\hat{E}}_{t}[\xi \eta ]=\xi
\mathbb{\hat{E}}_{t}[\eta]$;

\item[(v)] Consistency: For $t\leq s\leq r$, we have $\mathbb{\hat{E}}_{s}[%
\mathbb{\hat{E}}_{r}[\xi]]=\mathbb{\hat{E}}_{s}[\xi]$.

\item[(vi)] $\mathbb{\hat {E}}[\mathbb{\hat{E}}_{t}[\xi]]=\mathbb{\hat{E}}%
[\xi]$.
\end{description}
\end{proposition}

\begin{proof}
(i) is direct consequence of Lemma 2.9. (ii)-(v) are obvious from the
definition. We only prove $\mathbb{\hat{E}}[\mathbb{\hat{E}}_{t}[\xi]]=%
\mathbb{\hat{E}}[\xi]$ for $\xi$ which is bounded and positive.

Step 1. For $\xi=\sum_{i=1}^{N}I_{K_{i}}\eta_{i}$, where $K_{i}$, $%
i=1,\ldots,N$, are disjoint compact sets and $\eta_{i}\geq0$, we can choose $%
\varphi_{m}^{i}\in C_{b}(\Omega_{t})$ such that $\varphi_{m}^{i}\downarrow
K_{i}$ and $\varphi_{m}^{i}\varphi_{m}^{j}=0$ for $i\not =j$. By the same
analysis as that in Lemma \ref{lemA.7}, we can get $\mathbb{\hat{E}}%
[\sum_{i=1}^{N}I_{K_{i}}\mathbb{\hat{E}}_{t}[\eta_{i}]]=\mathbb{\hat{E}}%
[\sum_{i=1}^{N}I_{K_{i}}\eta_{i}]$.

Step 2. For $\xi=\sum_{i=1}^{N}I_{A_{i}}\eta_{i}$, where $A_{i}$, $%
i=1,\ldots,N$, are disjoint sets and $\eta_{i}\geq0$. For each fixed $P\in
\mathcal{P}$, we can choose compact sets $K_{m}^{i}$ such that $%
K_{m}^{i}\uparrow$ and $P(A_{i}-K_{m}^{i})\downarrow0$, then
\begin{align*}
E_{P}[\sum_{i=1}^{N}I_{A_{i}}\mathbb{\hat{E}}_{t}[\eta_{i}]] &
=\lim_{m\rightarrow \infty}E_{P}[\sum_{i=1}^{N}I_{K_{m}^{i}}\mathbb{\hat{E}}%
_{t}[\eta_{i}]] \\
& \leq \lim_{m\rightarrow \infty}\mathbb{\hat{E}}[\sum_{i=1}^{N}I_{K_{m}^{i}}%
\mathbb{\hat{E}}_{t}[\eta_{i}]] \\
& =\lim_{m\rightarrow \infty}\mathbb{\hat{E}}[\sum_{i=1}^{N}I_{K_{m}^{i}}%
\eta_{i}] \\
& \leq \mathbb{\hat{E}}[\sum_{i=1}^{N}I_{A_{i}}\eta_{i}].
\end{align*}
It follows that $\mathbb{\hat{E}}[\sum_{i=1}^{N}I_{A_{i}}\mathbb{\hat{E}}%
_{t}[\eta_{i}]]\leq \mathbb{\hat{E}}[\sum_{i=1}^{N}I_{A_{i}}\eta_{i}]$.
Similarly we can prove $\mathbb{\hat{E}}[\sum_{i=1}^{N}I_{A_{i}}\eta_{i}]%
\leq $ $\mathbb{\hat{E}}[\sum_{i=1}^{N}I_{A_{i}}\mathbb{\hat{E}}%
_{t}[\eta_{i}]]$.
\end{proof}

Let $\mathbb{L}_{G}^{p,t}(\Omega_T)$ be the completion of $\mathbb{L}%
_{G}^{0,p,t}(\Omega_T)$ under the norm $(\mathbb{\hat{E}}[|%
\cdot|^{p}])^{1/p} $. Clearly, the conditional $G$-expectation can be
extended continuously to $\mathbb{L}_{G}^{p,t}(\Omega_T)$.

Set%
\begin{equation*}
\mathbb{M}^{p,0}(0,T):=\{
\eta_{t}=\sum_{i=0}^{N-1}\xi_{t_{i}}I_{[t_{i},t_{i+1})}(t):0=t_{0}<%
\cdots<t_{N}=T,\xi_{t_{i}}\in \mathbb{L}^{p}(\Omega_{t_{i}})\}.
\end{equation*}
For $p\geq1$, we denote by $\mathbb{M}^{p}(0,T)$, $\mathbb{H}^{p}(0,T)$, $%
\mathbb{S}^{p}(0,T)$ the completion of $\mathbb{M}^{p,0}(0,T)$ under the
norm $||\eta||_{\mathbb{M}^{p}}:=(\mathbb{\hat{E}}[\int_{0}^{T}|%
\eta_{t}|^{p}dt])^{1/p}$, $||\eta||_{\mathbb{H}^{p}}:=\{ \mathbb{\hat{E}}%
[(\int_{0}^{T}|\eta_{t}|^{2}dt)^{p/2}]\}^{1/p}$, $||\eta||_{\mathbb{D}%
^{p}}:=(\mathbb{\hat{E}}[\sup_{t\in \lbrack0,T]}|\eta_{t}|^{p}])^{1/p}$
respectively. Following Li and Peng \cite{L-P}, for each $\eta \in \mathbb{H}%
^{p}(0,T)$ with $p\geq1$, we can define It\^{o}'s integral $\int
_{0}^{T}\eta_{s}dB_{s}$. Moreover, by Proposition 2.10 in \cite{L-P} and
classical Burkholder-Davis-Gundy Inequality, the following properties hold.

\begin{proposition}
\label{proA.5} For each $\eta,\theta \in \mathbb{H}^{\alpha}(0,T)$ with $%
\alpha \geq1$ and $p>0$, $\xi \in \mathbb{L}^{\infty}(\Omega_{t})$, we have%
\begin{equation*}
\mathbb{\hat{E}}[\int_{0}^{T}\eta_{s}dB_{s}]=0,
\end{equation*}%
\begin{equation*}
\ \underline{\sigma}^{p}c_{p}\mathbb{\hat{E}}[(\int_{0}^{T}|%
\eta_{s}|^{2}ds)^{p/2}]\leq \mathbb{\hat{E}}[\sup_{t\in
\lbrack0,T]}|\int_{0}^{t}\eta_{s}dB_{s}|^{p}]\leq \bar{\sigma}^{p}C_{p}%
\mathbb{\hat{E}}[(\int_{0}^{T}|\eta_{s}|^{2}ds)^{p/2}],
\end{equation*}%
\begin{equation*}
\int_{t}^{T}(\xi \eta_{s}+\theta_{s})dB_{s}=\xi
\int_{t}^{T}\eta_{s}dB_{s}+\int_{t}^{T}\theta_{s}dB_{s},
\end{equation*}
where $0<c_{p}<C_{p}<\infty$ are constants.
\end{proposition}

\begin{definition}
\label{def2.9} A process $\{M_{t}\}$ with values in $L_{G}^{1}(\Omega_{T})$
is called a $G$-martingale if $\mathbb{\hat{E}}_{s}[M_{t}]=M_{s}$ for any $%
s\leq t$.
\end{definition}

For $\xi \in L_{ip}(\Omega_{T})$, let $\mathcal{E}[\xi]=\mathbb{\hat{E}}%
[\sup_{t\in \lbrack0,T]}\mathbb{\hat{E}}_{t}[\xi]]$, where $\mathbb{\hat{E}}$
is the $G$-expectation. For convenience, we call $\mathcal{E}$ $G$%
-evaluation.

For $p\geq1$ and $\xi \in L_{ip}(\Omega_{T})$, define $\Vert \xi \Vert _{p,%
\mathcal{E}}=\{ \mathcal{E}[|\xi|^{p}]\}^{1/p}$ and denote by $L_{\mathcal{E}%
}^{p}(\Omega_{T})$ the completion of $L_{ip}(\Omega_{T})$ under the norm $%
\Vert \cdot \Vert_{p,\mathcal{E}}$.

Let $S_{G}^{0}(0,T)=\{h(t,B_{t_{1}\wedge t},\cdot \cdot \cdot,B_{t_{n}\wedge
t}):t_{1},\ldots,t_{n}\in \lbrack0,T],h\in C_{b,Lip}(\mathbb{R}^{n+1})\}$.
For $p\geq1$ and $\eta \in S_{G}^{0}(0,T)$, set $\Vert \eta
\Vert_{D_{G}^{p}}=\{ \mathbb{\hat{E}}[\sup_{t\in
\lbrack0,T]}|\eta_{t}|^{p}]\}^{\frac{1}{p}}$. Denote by $S_{G}^{p}(0,T)$ the
completion of $S_{G}^{0}(0,T)$ under the norm $\Vert \cdot \Vert_{S_{G}^{p}}$%
.

The following estimate will be frequently used in this paper.

\begin{theorem}
\label{the2.10} (\cite{Song11}) For any $\alpha \geq1$ and $\delta>0$, we
have $L_{G}^{\alpha+\delta}(\Omega_{T})\subset L_{\mathcal{E}%
}^{\alpha}(\Omega _{T})$. More precisely, for any $1<\gamma<\beta:=(\alpha+%
\delta)/\alpha$, $\gamma \leq2$ and for all $\xi \in L_{ip}(\Omega_{T})$, we
have%
\begin{equation}
\mathbb{\hat{E}}[\sup_{t\in \lbrack0,T]}\mathbb{\hat{E}}_{t}[|\xi|^{\alpha
}]]\leq C\{(\mathbb{\hat{E}}[|\xi|^{\alpha+\delta}])^{\alpha/(\alpha+\delta
)}+(\mathbb{\hat{E}}[|\xi|^{\alpha+\delta}])^{1/\gamma}\},  \label{e2song}
\end{equation}
where $C=\frac{\gamma}{\gamma-1}(1+14\sum_{i=1}^{\infty}i^{-\beta/\gamma})$.
\end{theorem}

\begin{remark}
\label{remn2.11} By $\frac{\alpha}{\alpha+\delta}<\frac{1}{\gamma}<1$, we
have%
\begin{equation*}
\mathbb{\hat{E}}[\sup_{t\in \lbrack0,T]}\mathbb{\hat{E}}_{t}[|\xi|^{\alpha
}]]\leq2C\{(\mathbb{\hat{E}}[|\xi|^{\alpha+\delta}])^{\alpha/(\alpha+\delta
)}+\mathbb{\hat{E}}[|\xi|^{\alpha+\delta}]\}.
\end{equation*}
Set $C_{1}=2\inf \{ \frac{\gamma}{\gamma-1}(1+14\sum_{i=1}^{\infty}i^{-%
\beta/\gamma}):1<\gamma<\beta,\gamma \leq2\}$, then
\begin{equation}
\mathbb{\hat{E}}[\sup_{t\in \lbrack0,T]}\mathbb{\hat{E}}_{t}[|\xi|^{\alpha
}]]\leq C_{1}\{(\mathbb{\hat{E}}[|\xi|^{\alpha+\delta}])^{\alpha
/(\alpha+\delta)}+\mathbb{\hat{E}}[|\xi|^{\alpha+\delta}]\},  \label{e2song1}
\end{equation}
where $C_{1}$ is a constant only depending on $\alpha$ and $\delta$.
\end{remark}

For readers' convenience, we list the main notations of this paper as
follows:

\begin{itemize}
\item The scalar product and norm of the Euclid space $\mathbb{R}^{n}$ are
respectively denoted by $\langle \cdot,\cdot \rangle$ and $|\cdot|$;

\item $L_{ip}(\Omega_{T}):=$\{$\varphi(B_{t_{1}},...,B_{t_{n}}):n\geq1$, $%
t_{1},...,t_{n}\in \lbrack0,T]$, $\varphi \in C_{b.Lip}(\mathbb{R}^{d\times
n})$\};

\item $\Vert \xi \Vert _{p,G}=(\mathbb{\hat{E}}[|\xi |^{p}])^{1/p}$, $\ \ \
\ \Vert \xi \Vert _{p,\mathcal{E}}=(\mathbb{\hat{E}}[\sup_{t\in \lbrack 0,T]}%
\mathbb{\hat{E}}_{t}[|\xi |^{p}]])^{1/p}$;

\item $L_{G}^{p}(\Omega _{T}):=$the completion of $L_{ip}(\Omega _{T})$
under $\Vert \cdot \Vert _{p,G}$;

\item $L_{\mathcal{E}}^{p}(\Omega _{T}):=$the completion of $L_{ip}(\Omega
_{T})$ under $\Vert \cdot \Vert _{p,\mathcal{E}}$;

\item $M_{G}^{0}(0,T):=$\{$\eta_{t}=\sum_{j=0}^{N-1}%
\xi_{j}I_{[t_{j},t_{j+1})}(t):0=t_{0}<\cdots<t_{N}=T$, $\xi_{i}\in
L_{ip}(\Omega_{t_{i}})$\};

\item $\Vert \eta \Vert _{M_{G}^{p}}=\{\mathbb{\hat{E}}[\int_{0}^{T}|\eta
_{s}|^{p}ds]\}^{1/p}$, $\ \ \ \ \ \ \ \Vert \eta \Vert _{H_{G}^{p}}=\{%
\mathbb{\hat{E}}[(\int_{0}^{T}|\eta _{s}|^{2}ds)^{p/2}]\}^{1/p}$;

\item $M_{G}^{p}(0,T):=$the completion of $M_{G}^{0}(0,T)$ under $\Vert
\cdot \Vert _{M_{G}^{p}}$;

\item $H_{G}^{p}(0,T):=$\{the completion of $M_{G}^{0}(0,T)$ under $\Vert
\cdot \Vert _{H_{G}^{p}}$\} for $p\geq 1$;

\item $\mathbb{L}^{p}(\Omega_{T}):=$\{$X\in \mathcal{B}(\Omega_{T}):\sup
_{P\in \mathcal{P}}E_{P}[|X|^{p}]<\infty$\}$\ $for$\ p\geq1$;

\item $\mathbb{M}^{p,0}(0,T):=$\{$\eta_{t}=\sum_{i=0}^{N-1}%
\xi_{t_{i}}I_{[t_{i},t_{i+1})}(t):0=t_{0}<\cdots<t_{N}=T$, $\xi_{t_{i}}\in
\mathbb{L}^{p}(\Omega_{t_{i}})$\};

\item $||\eta ||_{\mathbb{M}^{p}}:=(\mathbb{\hat{E}}[\int_{0}^{T}|\eta
_{t}|^{p}dt])^{1/p}$, $\ \ \ \ \ \ ||\eta ||_{\mathbb{H}^{p}}:=\{\mathbb{%
\hat{E}}[(\int_{0}^{T}|\eta _{t}|^{2}dt)^{p/2}]\}^{1/p}$;

\item $||\eta||_{\mathbb{S}^{p}}:=\{ \mathbb{\hat{E}}[\sup_{t\in \lbrack
0,T]}|\eta_{t}|^{p}]\}^{\frac{1}{p}}$;

\item $\mathbb{M}^{p}(0,T):=$the completion of $\mathbb{M}^{p,0}(0,T)$ under
$||\cdot ||_{\mathbb{M}^{p}}$;

\item $\mathbb{H}^{p}(0,T):=$the completion of $\mathbb{M}^{p,0}(0,T)$ under
$||\cdot ||_{\mathbb{H}^{p}}$;

\item $\mathbb{S}^{p}(0,T):=$the completion of $\mathbb{M}^{p,0}(0,T)$ under
$||\cdot ||_{\mathbb{S}^{p}}$;

\item $S_{G}^{0}(0,T)=$\{$h(t,B_{t_{1}\wedge t},\cdot \cdot
\cdot,B_{t_{n}\wedge t}):t_{1},\ldots,t_{n}\in \lbrack0,T]$, $h\in C_{b,Lip}(%
\mathbb{R}^{n+1})$\};

\item $\Vert \eta \Vert_{S_{G}^{p}}=\{ \mathbb{\hat{E}}[\sup_{t\in \lbrack
0,T]}|\eta_{t}|^{p}]\}^{\frac{1}{p}}$;

\item $S_{G}^{p}(0,T):=$the completion of $S_{G}^{0}(0,T)$ under $\Vert
\cdot \Vert _{S_{G}^{p}}$;

\item $\mathfrak{S}_{G}^{\alpha }(0,T):=$ the collection of processes $%
(Y,Z,K)$ such that $Y\in S_{G}^{\alpha }(0,T)$, $Z\in H_{G}^{\alpha }(0,T)$,
$K$ is a decreasing $G$-martingale with $K_{0}=0$ and $K_{T}\in
L_{G}^{\alpha }(\Omega _{T})$.
\end{itemize}

\section{A priori estimates}

For simplicity, we consider the $G$-expectation space $%
(\Omega_{T},L_{G}^{1}(\Omega_{T}),\mathbb{\hat{E}})$ with $%
\Omega_{T}=C_{0}([0,T],\mathbb{R})$ and $\overline{\sigma}^{2}=\mathbb{\hat{E%
}}[B_{1}^{2}]\geq-\mathbb{\hat{E}}[-B_{1}^{2}]=\underline{\sigma}^{2}>0$.
But our results and methods still hold for the case $d>1$.

We consider the following type of $G$-BSDEs for simplicity, and similar
estimates hold for equation (\ref{e1}).
\begin{equation}
Y_{t}=\xi+\int_{t}^{T}f(s,Y_{s},Z_{s})ds-%
\int_{t}^{T}Z_{s}dB_{s}-(K_{T}-K_{t}),  \label{e3}
\end{equation}
where
\begin{equation*}
f(t,\omega,y,z):[0,T]\times \Omega_{T}\times \mathbb{R}^{2}\rightarrow
\mathbb{R}
\end{equation*}
satisfies the following properties: There exists some $\beta>1$ such that

\begin{description}
\item[(H1)] for any $y,z$, $f(\cdot,\cdot,y,z)\in M_{G}^{\beta}(0,T)$;

\item[(H2)] $|f(t,\omega,y,z)-f(t,\omega,y^{\prime},z^{\prime})|\leq
L(|y-y^{\prime}|+|z-z^{\prime}|)$ for some $L>0$.
\end{description}

For simplicity, we denote by $\mathfrak{S}_{G}^{\alpha }(0,T)$ the
collection of processes $(Y,Z,K)$ such that $Y\in S_{G}^{\alpha }(0,T)$, $%
Z\in H_{G}^{\alpha }(0,T)$, $K$ is a decreasing $G$-martingale with $K_{0}=0$
and $K_{T}\in L_{G}^{\alpha }(\Omega _{T})$.

\begin{definition}
\label{def3.1} Let $\xi \in L_{G}^{\beta}(\Omega_{T})$ with $\beta>1$ and $f$
satisfy (H1) and (H2). A triplet of processes $(Y,Z,K)$ is called a solution
of equation (\ref{e3}) if for some $1<\alpha\leq\beta$ the following
properties hold:

\begin{description}
\item[(a)] $(Y,Z,K)\in \mathfrak{S}_{G}^{\alpha }(0,T)$;

\item[(b)] $Y_{t}=\xi+\int_{t}^{T}f(s,Y_{s},Z_{s})ds-%
\int_{t}^{T}Z_{s}dB_{s}-(K_{T}-K_{t})$.
\end{description}
\end{definition}

In order to get a priori estimates for the solution of equation (\ref{e3}),
we need the following lemmas.

\begin{lemma}
\label{lem3.1} Let $X\in S^\alpha_G(0,T)$ for some $\alpha>1$. Set
\begin{equation*}
X_{t}^{n}=\sum_{i=0}^{n-1}X_{t_{i}^{n}}I_{[t_{i}^{n},t_{i+1}^{n})}(t),
\end{equation*}
where $t_{i}^{n}=\frac{iT}{n}$, $i=0,\cdot \cdot \cdot,n$. Then%
\begin{equation}
\mathbb{\hat{E}}[\sup_{t\in
\lbrack0,T]}|X_{t}^{n}-X_{t}|^{\alpha}]\rightarrow0\text{ \ as }n\rightarrow
\infty.  \label{e32}
\end{equation}
\end{lemma}

\begin{proof}
For each given $n$, $m\geq1$, it is easy to check that
\begin{equation*}
\sup_{i\leq n-1}\sup_{t_{k}^{m}\in \lbrack
t_{i}^{n},t_{i+1}^{n}]}|B_{t_{k}^{m}}-B_{t_{i}^{n}}|^{\alpha}
\end{equation*}
is a convex function, then by Proposition 11 in Peng \cite{P07a}, we get%
\begin{equation*}
\mathbb{\hat{E}}[\sup_{i\leq n-1}\sup_{t_{k}^{m}\in \lbrack
t_{i}^{n},t_{i+1}^{n}]}|B_{t_{k}^{m}}-B_{t_{i}^{n}}|^{\alpha}]=E_{P_{\bar{%
\sigma}}}[\sup_{i\leq n-1}\sup_{t_{k}^{m}\in \lbrack
t_{i}^{n},t_{i+1}^{n}]}|B_{t_{k}^{m}}-B_{t_{i}^{n}}|^{\alpha}],
\end{equation*}
where $P_{\bar{\sigma}}$ is a Wiener measure on $\Omega_{T}$ such that $%
E_{P_{\bar{\sigma}}}[B_{1}^{2}]=\bar{\sigma}^{2}$. Noting that%
\begin{equation*}
\sup_{i\leq n-1}\sup_{t_{k}^{m}\in \lbrack
t_{i}^{n},t_{i+1}^{n}]}|B_{t_{k}^{m}}-B_{t_{i}^{n}}|^{\alpha}\uparrow
\sup_{i\leq n-1}\sup_{t\in \lbrack
t_{i}^{n},t_{i+1}^{n}]}|B_{t}-B_{t_{i}^{n}}|^{\alpha}\text{ as }m\uparrow
\infty,
\end{equation*}
we have%
\begin{equation*}
\mathbb{\hat{E}}[\sup_{i\leq n-1}\sup_{t\in \lbrack
t_{i}^{n},t_{i+1}^{n}]}|B_{t}-B_{t_{i}^{n}}|^{\alpha}]=E_{P_{\bar{\sigma}%
}}[\sup_{i\leq n-1}\sup_{t\in \lbrack
t_{i}^{n},t_{i+1}^{n}]}|B_{t}-B_{t_{i}^{n}}|^{\alpha }]\rightarrow0.
\end{equation*}
From this we can get $\mathbb{\hat{E}}[\sup_{t\in
\lbrack0,T]}|\eta_{t}-\eta_{t}^{n}|^{\alpha}]\rightarrow0$ for each $\eta
\in S_{G}^{0}(0,T)$. By the definition of $S^\alpha_G(0,T)$, we can choose a
sequence $(\eta ^{m})_{m=1}^{\infty}\subset S_{G}^{0}(0,T)$ such that $%
\mathbb{\hat{E}}[\sup_{t\in
\lbrack0,T]}|X_{t}-\eta_{t}^{m}|^{\alpha}]\rightarrow0$ as $m\rightarrow
\infty$. Note that
\begin{equation*}
\sup_{t\in \lbrack0,T]}|X_{t}-X_{t}^{n}|\leq2\sup_{t\in
\lbrack0,T]}|X_{t}-\eta_{t}^{m}|+\sup_{t\in
\lbrack0,T]}|\eta_{t}^{m}-(\eta^{m})_{t}^{n}|,
\end{equation*}
then we obtain (\ref{e32}) by letting $n\rightarrow \infty$ first and then $%
m\rightarrow \infty$.
\end{proof}

\begin{lemma}
\label{lem3.2} Let $X_{t}$, $X_{t}^{n}$ be as in Lemma \ref{lem3.1} and $%
\alpha^{\ast}=\frac{\alpha}{\alpha-1}$. Assume that $K$ is a decreasing $G$%
-martingale with $K_{0}=0$ and $K_{T}\in L_{G}^{\alpha^{\ast}}(\Omega_{T})$.
Then we have
\begin{equation*}
\mathbb{\hat{E}}[\sup_{t\in \lbrack0,T]}|\int_{0}^{t}X_{s}^{n}dK_{s}-\int
_{0}^{t}X_{s}dK_{s}|]\rightarrow0\text{ \ as }n\rightarrow \infty.
\end{equation*}
\end{lemma}

\begin{proof}
\begin{align*}
& \sup_{t\in
\lbrack0,T]}|\int_{0}^{t}X_{s}^{n}dK_{s}-\int_{0}^{t}X_{s}dK_{s}| \\
& \leq-\int_{0}^{T}|X_{s}^{n}-X_{s}|dK_{s} \\
& \leq \sup_{s\in \lbrack0,T]}|X_{s}^{n}-X_{s}|(-K_{T}).
\end{align*}
So we have
\begin{equation*}
\mathbb{\hat{E}}[\sup_{t\in \lbrack0,T]}|\int_{0}^{t}X_{s}^{n}dK_{s}-\int
_{0}^{t}X_{s}dK_{s}|]\leq \Vert \sup_{s\in
\lbrack0,T]}|X_{s}^{n}-X_{s}|\Vert_{L_{G}^{\alpha}}\Vert
K_{T}\Vert_{L^{\alpha^{\ast}}}\rightarrow0
\end{equation*}
as $n\rightarrow \infty$.
\end{proof}

\begin{lemma}
\label{lem3.3} Let $X\in S^\alpha_G(0,T)$ for some $\alpha>1$ and $%
\alpha^{\ast}=\frac{\alpha}{\alpha-1}$. Assume that $K^{j}$, $j=1,2$, are
two decreasing $G$-martingales with $K_{0}^{j}=0$ and $K_{T}^{j}\in
L_{G}^{\alpha^{\ast}}(\Omega_{T})$. Then the process defined by
\begin{equation*}
\int_{0}^{t}X_{s}^{+}dK_{s}^{1}+\int_{0}^{t}X_{s}^{-}dK_{s}^{2}
\end{equation*}
is also a decreasing $G$-martingale.
\end{lemma}

\begin{proof}
Let $X^{n}$ be as in Lemma \ref{lem3.1}. By Lemma \ref{lem3.2}, it suffices
to prove that the process
\begin{equation*}
\int_{0}^{t}(X_{s}^{n})^{+}dK_{s}^{1}+\int_{0}^{t}(X_{s}^{n})^{-}dK_{s}^{2}
\end{equation*}
is a $G$-martingale. By properties of conditional $G$-expectation, we have,
for any $t\in \lbrack t_{i}^{n},t_{i+1}^{n}]$,
\begin{align*}
& \mathbb{\hat{E}}%
_{t}[X_{t_{i}^{n}}^{+}(K_{t_{i+1}^{n}}^{1}-K_{t_{i}^{n}}^{1})+X_{t_{i}^{n}}^{-}(K_{t_{i+1}^{n}}^{2}-K_{t_{i}^{n}}^{2})]
\\
& =X_{t_{i}^{n}}^{+}\mathbb{\hat{E}}%
_{t}[K_{t_{i+1}^{n}}^{1}-K_{t_{i}^{n}}^{1}]+X_{t_{i}^{n}}^{-}\mathbb{\hat{E}}%
_{t}[K_{t_{i+1}^{n}}^{2}-K_{t_{i}^{n}}^{2}] \\
&
=X_{t_{i}^{n}}^{+}(K_{t}^{1}-K_{t_{i}^{n}}^{1})+X_{t_{i}^{n}}^{-}(K_{t}^{2}-K_{t_{i}^{n}}^{2}).
\end{align*}
From this we obtain that $\int_{0}^{t}(X_{s}^{n})^{+}dK_{s}^{1}+%
\int_{0}^{t}(X_{s}^{n})^{-}dK_{s}^{2}$ is a $G$-martingale.
\end{proof}

Now we give a priori estimates for the solution of equation (\ref{e3}). For
this purpose, a weaker version of condition (H2) is enough.

\begin{description}
\item[(H2')] $|f(t,\omega,y,z)-f(t,\omega,y^{\prime},z^{\prime})|\leq
L^{w}(|y-y^{\prime}|+|z-z^{\prime}|+\varepsilon)$ for some $%
L^{w},\varepsilon>0$.
\end{description}

In the following three propositions, $C_{\alpha}$ will always designate a
constant depending on $\alpha,T,L^{w},\underline{\sigma}$, which may vary
from line to line.

\begin{proposition}
\label{pron1} Let $f$ satisfy (H1) and (H2'). Assume%
\begin{equation*}
Y_{t}=\xi+\int_{t}^{T}f(s,Y_{s},Z_{s})ds-%
\int_{t}^{T}Z_{s}dB_{s}-(K_{T}-K_{t}),
\end{equation*}
where $Y\in \mathbb{S}^{\alpha}(0,T)$, $Z\in \mathbb{H}^{\alpha}(0,T)$, $K$
is a decreasing process with $K_{0}=0$ and $K_{T}\in \mathbb{L}%
^{\alpha}(\Omega _{T})$ for some $\alpha>1$. Then there exists a constant $%
C_{\alpha}:=C(\alpha,T,\underline{\sigma},L^{w})>0$ such that%
\begin{equation}
\mathbb{\hat{E}}[(\int_{0}^{T}|Z_{s}|^{2}ds)^{\frac{\alpha}{2}}]\leq
C_{\alpha}\{ \mathbb{\hat{E}}[\sup_{t\in \lbrack0,T]}|Y_{t}|^{\alpha }]+(%
\mathbb{\hat{E}}[\sup_{t\in \lbrack0,T]}|Y_{t}|^{\alpha}])^{\frac{1}{2}}(%
\mathbb{\hat{E}}[(\int_{0}^{T}f_{s}^{0}ds)^{\alpha}])^{\frac{1}{2}}\},
\label{e37}
\end{equation}%
\begin{equation}
\mathbb{\hat{E}}[|K_{T}|^{\alpha}]\leq C_{\alpha}\{ \mathbb{\hat{E}}%
[\sup_{t\in \lbrack0,T]}|Y_{t}|^{\alpha}]+\mathbb{\hat{E}}%
[(\int_{0}^{T}f_{s}^{0}ds)^{\alpha}]\},  \label{e34}
\end{equation}
where $f_{s}^{0}=|f(s,0,0)|+L^{w}\varepsilon$.
\end{proposition}

\begin{proof}
Applying It\^{o}'s formula to $|Y_{t}|^{2}$, we have%
\begin{equation*}
|Y_{0}|^{2}+\int_{0}^{T}|Z_{s}|^{2}d\langle B\rangle_{s}=|\xi|^{2}+\int
_{0}^{T}2Y_{s}f(s)ds-\int_{0}^{T}2Y_{s}Z_{s}dB_{s}-\int_{0}^{T}2Y_{s}dK_{s},
\end{equation*}
where $f(s)=f(s,Y_{s},Z_{s})$. Then%
\begin{equation*}
(\int_{0}^{T}|Z_{s}|^{2}d\langle B\rangle_{s})^{\frac{\alpha}{2}}\leq
C_{\alpha}\{|\xi|^{\alpha}+|\int_{0}^{T}Y_{s}f(s)ds|^{\frac{\alpha}{2}%
}+|\int_{0}^{T}Y_{s}Z_{s}dB_{s}|^{\frac{\alpha}{2}}+|%
\int_{0}^{T}Y_{s}dK_{s}|^{\frac{\alpha}{2}}\}.
\end{equation*}
By Proposition \ref{proA.5} and simple calculation, we can obtain%
\begin{equation}
\mathbb{\hat{E}}[(\int_{0}^{T}|Z_{s}|^{2}ds)^{\frac{\alpha}{2}}]\leq
C_{\alpha}\{ \Vert Y\Vert_{\mathbb{S}^{\alpha}}^{\alpha}+\Vert Y\Vert _{%
\mathbb{S}^{\alpha}}^{\frac{\alpha}{2}}[(\mathbb{\hat{E}}[|K_{T}|^{\alpha
}])^{\frac{1}{2}}+(\mathbb{\hat{E}}[(\int_{0}^{T}f_{s}^{0}ds)^{\alpha }])^{%
\frac{1}{2}}]\}.  \label{e35}
\end{equation}
On the other hand,%
\begin{equation*}
K_{T}=\xi-Y_{0}+\int_{0}^{T}f(s)ds-\int_{0}^{T}Z_{s}dB_{s}.
\end{equation*}
By simple calculation, we get%
\begin{equation}
\mathbb{\hat{E}}[|K_{T}|^{\alpha}]\leq C_{\alpha}\{ \Vert Y\Vert _{\mathbb{S}%
^{\alpha}}^{\alpha}+\mathbb{\hat{E}}[(\int_{0}^{T}|Z_{s}|^{2}ds)^{\alpha/2}]+%
\mathbb{\hat{E}}[(\int_{0}^{T}f_{s}^{0}ds)^{\alpha}]\}.  \label{e36}
\end{equation}
By (\ref{e35}) and (\ref{e36}), it is easy to get (\ref{e37}) and (\ref{e34}%
).
\end{proof}

\begin{remark}
\label{remn2}In this proposition, we do not assume that $(Y,Z,K)$ belong to $%
\mathfrak{S}_{G}^{\alpha }(0,T)$.
\end{remark}

\begin{proposition}
\label{pro3.4} Let $\xi \in L_{G}^{\beta }(\Omega _{T})$ with $\beta >1$ and
$f$ satisfy (H1) and (H2'). Assume that $(Y,Z,K)\in \mathfrak{S}_{G}^{\alpha
}(0,T)$ for some $1<\alpha <\beta $ is a solution of equation (\ref{e3}).
Then

\begin{description}
\item[(i)] There exists a constant $C_{\alpha}:=C(\alpha,T,\underline{\sigma
},L^{w})>0$ such that%
\begin{equation}
|Y_{t}|^{\alpha}\leq C_{\alpha}\mathbb{\hat{E}}_{t}[|\xi|^{\alpha}+\int
_{t}^{T}|f_{s}^{0}|^{\alpha}ds],  \label{e38}
\end{equation}%
\begin{equation}
\mathbb{\hat{E}}[\sup_{t\in \lbrack0,T]}|Y_{t}|^{\alpha}]\leq C_{\alpha }%
\mathbb{\hat{E}}[\sup_{t\in \lbrack0,T]}\mathbb{\hat{E}}_{t}[|\xi|^{\alpha
}+\int_{0}^{T}|f_{s}^{0}|^{\alpha}ds]],  \label{e39}
\end{equation}
where $f_{s}^{0}=|f(s,0,0)|+L^{w}\varepsilon$.

\item[(ii)] For any given $\alpha^{\prime }$ with $\alpha<\alpha^{\prime
}<\beta$, there exists a constant $C_{\alpha, \alpha^{\prime }}$ depending
on $\alpha $, $\alpha^{\prime }$, $T$, $\underline{\sigma}$, $L^{w}$ such
that%
\begin{align}
\mathbb{\hat{E}}[\sup_{t\in \lbrack0,T]}|Y_{t}|^{\alpha}] & \leq C_{\alpha,
\alpha^{\prime }}\{ \mathbb{\hat{E}}[\sup_{t\in \lbrack0,T]}\mathbb{\hat{E}}%
_{t}[|\xi|^{\alpha }]]  \notag \\
& +(\mathbb{\hat{E}}[\sup_{t\in \lbrack0,T]}\mathbb{\hat{E}}%
_{t}[(\int_{0}^{T}f_{s}^{0}ds)^{\alpha^{\prime }}]])^{\frac{\alpha}{%
\alpha^{\prime }}}+\mathbb{\hat{E}}[\sup_{t\in \lbrack0,T]}\mathbb{\hat{E}}%
_{t}[(\int_{0}^{T}f_{s}^{0}ds)^{\alpha^{\prime }}]]\}.  \label{e310}
\end{align}
\end{description}
\end{proposition}

\begin{proof}
For any $\gamma,\epsilon>0$, set $\tilde{Y}_{t}=|Y_{t}|^{2}+\epsilon_{%
\alpha} $, where $\epsilon_{\alpha}=\epsilon(1-\alpha/2)^{+}$, applying It%
\^{o}'s formula to $\tilde{Y}_{t}{}^{\alpha/2}e^{\gamma t}$, we have

\begin{align}
& \tilde{Y}_{t}{}^{\alpha/2}e^{\gamma t}+\gamma \int_{t}^{T}e^{\gamma s}%
\tilde{Y}_{s}{}^{\alpha/2}ds+\frac{\alpha}{2}\int_{t}^{T}e^{\gamma s}\tilde {%
Y}_{s}{}^{\alpha/2-1}Z_{s}^{2}d\langle B\rangle_{s}  \notag \\
& =(|\xi|^{2}+\epsilon_{\alpha})^{\alpha/2}e^{\gamma T}+\alpha(1-\frac {%
\alpha}{2})\int_{t}^{T}e^{\gamma s}\tilde{Y}_{s}{}^{%
\alpha/2-2}Y_{s}^{2}Z_{s}^{2}d\langle B\rangle_{s}  \notag \\
& +\int_{t}^{T}\alpha e^{\gamma s}\tilde{Y}_{s}{}^{\alpha/2-1}Y_{s}f(s)ds-%
\int_{t}^{T}\alpha e^{\gamma s}\tilde{Y}_{s}{}^{%
\alpha/2-1}(Y_{s}Z_{s}dB_{s}+Y_{s}dK_{s})  \notag \\
& \leq(|\xi|^{2}+\epsilon_{\alpha})^{\alpha/2}e^{\gamma T}+\alpha (1-\frac{%
\alpha}{2})\int_{t}^{T}e^{\gamma s}\tilde{Y}_{s}{}^{\alpha/2-1}Z_{s}^{2}d%
\langle B\rangle_{s}  \notag \\
& +\int_{t}^{T}\alpha e^{\gamma s}\tilde{Y}_{s}{}^{%
\alpha/2-1/2}|f(s)|ds-(M_{T}-M_{t}),  \label{e311}
\end{align}
where $f(s)=f(s,Y_{s},Z_{s})$ and
\begin{equation*}
M_{t}=\int_{0}^{t}\alpha e^{\gamma s}\tilde{Y}_{s}{}^{%
\alpha/2-1}Y_{s}Z_{s}dB_{s}+\int_{0}^{t}\alpha e^{\gamma s}\tilde{Y}%
_{s}{}^{\alpha/2-1}Y_{s}^{+}dK_{s}.
\end{equation*}
From the assumption of $f$, we have
\begin{align}
& \int_{t}^{T}\alpha e^{\gamma s}\tilde{Y}_{s}{}^{\alpha/2-1/2}|f(s)|ds
\notag \\
& \leq \int_{t}^{T}\alpha e^{\gamma s}\tilde{Y}_{s}{}^{%
\alpha/2-1/2}(f_{s}^{0}+L^{w}|Y_{s}|+L^{w}|Z_{s}|)ds  \notag \\
& \leq(\alpha L^{w}+\frac{\alpha(L^{w})^{2}}{\underline{\sigma}^{2}(\alpha-1)%
})\int_{t}^{T}e^{\gamma s}\tilde{Y}_{s}{}^{\alpha/2}ds+\frac {%
\alpha(\alpha-1)}{4}\int_{t}^{T}e^{\gamma s}\tilde{Y}_{s}{}^{%
\alpha/2-1}Z_{s}^{2}d\langle B\rangle_{s}  \notag \\
& +\int_{t}^{T}\alpha e^{\gamma s}\tilde{Y}_{s}{}^{%
\alpha/2-1/2}|f_{s}^{0}|ds.  \label{e312}
\end{align}
\textbf{(i)} By Young's inequality, we have%
\begin{equation}
\int_{t}^{T}\alpha e^{\gamma s}\tilde{Y}_{s}{}^{\alpha/2-1/2}|f_{s}^{0}|ds%
\leq(\alpha-1)\int_{t}^{T}e^{\gamma s}\tilde{Y}_{s}{}^{\alpha/2}ds+\int
_{t}^{T}e^{\gamma s}|f_{s}^{0}|^{\alpha}ds.  \label{e313}
\end{equation}
By (\ref{e311}), (\ref{e312}) and (\ref{e313}), we have
\begin{align*}
& \tilde{Y}_{t}{}^{\alpha/2}e^{\gamma t}+(\gamma-\tilde{\alpha}%
)\int_{t}^{T}e^{\gamma s}\tilde{Y}_{s}{}^{\alpha/2}ds+\frac{\alpha(\alpha-1)%
}{4}\int_{t}^{T}e^{\gamma s}\tilde{Y}_{s}{}^{\alpha/2-1}Z_{s}^{2}d\langle
B\rangle_{s} \\
& \leq(|\xi|^{2}+\epsilon_{\alpha})^{\alpha/2}e^{\gamma
T}+\int_{t}^{T}e^{\gamma s}|f_{s}^{0}|^{\alpha}ds-(M_{T}-M_{t}),
\end{align*}
where $\tilde{\alpha}=\alpha L^{w}+\alpha+\frac{\alpha(L^{w})^{2}}{%
\underline{\sigma}^{2}(\alpha-1)}-1$. Setting $\gamma=\tilde{\alpha}+1$, we
have
\begin{align*}
& \tilde{Y}_{t}{}^{\alpha/2}e^{\gamma t}+M_{T}-M_{t} \\
& \leq(|\xi|^{2}+\epsilon_{\alpha})^{\alpha/2}e^{\gamma
T}+\int_{t}^{T}e^{\gamma s}|f_{s}^{0}|^{\alpha}ds.
\end{align*}
By Lemma \ref{lem3.3}, $M_{t}$ is a $G$-martingale, so we have
\begin{equation*}
\tilde{Y}_{t}{}^{\alpha/2}e^{\gamma t}\leq \mathbb{\hat{E}}_{t}[(|\xi
|^{2}+\epsilon_{\alpha})^{\alpha/2}e^{\gamma T}+\int_{t}^{T}e^{\gamma
s}|f_{s}^{0}|^{\alpha}ds].
\end{equation*}
By letting $\epsilon \downarrow0$, there exists a constant $C_{\alpha
}:=C_{\alpha}(T,L^{w},\underline{\sigma})$ such that
\begin{equation*}
|Y_{t}|^{\alpha}\leq C_{\alpha}\mathbb{\hat{E}}_{t}[|\xi|^{\alpha}+\int
_{t}^{T}|f_{s}^{0}|^{\alpha}ds].
\end{equation*}
It follows that
\begin{equation*}
\mathbb{\hat{E}}[\sup_{t\in \lbrack0,T]}|Y_{t}|^{\alpha}]\leq C_{\alpha }%
\mathbb{\hat{E}}[\sup_{t\in \lbrack0,T]}\mathbb{\hat{E}}_{t}[|\xi|^{\alpha
}+\int_{0}^{T}|f_{s}^{0}|^{\alpha}ds]].
\end{equation*}
\textbf{(ii)} By (\ref{e311}) and (\ref{e312}) and setting $\gamma=\alpha
L^{w}+\frac {\alpha(L^{w})^{2}}{\underline{\sigma}^{2}(\alpha-1)}+1$, then
we get%
\begin{equation*}
\tilde{Y}_{t}{}^{\alpha/2}e^{\gamma t}\leq \mathbb{\hat{E}}_{t}[(|\xi
|^{2}+\epsilon_{\alpha})^{\alpha/2}e^{\gamma T}+\int_{t}^{T}\alpha e^{\gamma
s}\tilde{Y}_{s}{}^{\alpha/2-1/2}f_{s}^{0}ds].
\end{equation*}
By letting $\epsilon \downarrow0$, we get%
\begin{equation}
|Y_{t}|^{\alpha}\leq C_{\alpha}\mathbb{\hat{E}}_{t}[|\xi|^{\alpha}+\int
_{t}^{T}|Y_{s}|^{\alpha-1}f_{s}^{0}ds].  \label{e314}
\end{equation}
From this we get%
\begin{align}
|Y_{t}|^{\alpha} & \leq C_{\alpha}\{ \mathbb{\hat{E}}_{t}[|\xi|^{\alpha }]+%
\mathbb{\hat{E}}_{t}[\sup_{s\in
\lbrack0,T]}|Y_{s}|^{\alpha-1}\int_{0}^{T}f_{s}^{0}ds]\}  \notag \\
& \leq C_{\alpha}\{ \mathbb{\hat{E}}_{t}[|\xi|^{\alpha}]+(\mathbb{\hat{E}}%
_{t}[\sup_{s\in \lbrack0,T]}|Y_{s}|^{(\alpha-1)\alpha^{\prime \ast}}])^{%
\frac {1}{\alpha^{\prime \ast}}}(\mathbb{\hat{E}}_{t}[(%
\int_{0}^{T}f_{s}^{0}ds)^{\alpha^{\prime }}])^{\frac{1}{\alpha^{\prime }}}\},
\label{e315}
\end{align}
where $\alpha^{\prime \ast}=\frac{\alpha^{\prime }}{\alpha^{\prime }-1}$.
Thus we obtain%
\begin{equation*}
\mathbb{\hat{E}}[\sup_{t\in \lbrack0,T]}|Y_{t}|^{\alpha}]\leq
C_{\alpha}\{||\xi||_{\alpha,\mathcal{E}}^{\alpha}+||\sup_{s\in
\lbrack0,T]}|Y_{s}|^{\alpha-1}||_{\alpha^{\prime \ast},\mathcal{E}%
}||\int_{0}^{T}f_{s}^{0}ds||_{\alpha^{\prime },\mathcal{E}}\}.
\end{equation*}
It is easy to check that $(\alpha-1)\alpha^{\prime \ast}<\alpha$, then by (%
\ref{e2song1}) there exists a constant $C$ only depending on $\alpha$ and $%
\alpha^{\prime }$ such that%
\begin{equation*}
||\sup_{s\in \lbrack0,T]}|Y_{s}|^{\alpha-1}||_{\alpha^{\prime \ast},\mathcal{%
E}}\leq C\{(\mathbb{\hat{E}}[\sup_{t\in \lbrack0,T]}|Y_{t}|^{\alpha}])^{%
\frac{\alpha-1}{\alpha}}+(\mathbb{\hat{E}}[\sup_{t\in
\lbrack0,T]}|Y_{t}|^{\alpha}])^{\frac{1}{\alpha^{\prime \ast}}}\}.
\end{equation*}
By Young's inequality, we have%
\begin{equation*}
CC_{\alpha}(\mathbb{\hat{E}}[\sup_{t\in \lbrack0,T]}|Y_{t}|^{\alpha}])^{%
\frac{\alpha-1}{\alpha}}||\int_{0}^{T}f_{s}^{0}ds||_{\alpha^{\prime },%
\mathcal{E}}\leq \frac{1}{4}\mathbb{\hat{E}}[\sup_{t\in
\lbrack0,T]}|Y_{t}|^{\alpha}]+C_{1}C_{\alpha}||\int_{0}^{T}f_{s}^{0}ds||_{%
\alpha _{1},\mathcal{E}}^{\alpha}
\end{equation*}
and%
\begin{equation*}
CC_{\alpha}(\mathbb{\hat{E}}[\sup_{t\in \lbrack0,T]}|Y_{t}|^{\alpha}])^{%
\frac{1}{\alpha^{\prime \ast}}}||\int_{0}^{T}f_{s}^{0}ds||_{\alpha _{1},%
\mathcal{E}}\leq \frac{1}{4}\mathbb{\hat{E}}[\sup_{t\in
\lbrack0,T]}|Y_{t}|^{\alpha}]+C_{1}C_{\alpha}||\int_{0}^{T}f_{s}^{0}ds||_{%
\alpha _{1},\mathcal{E}}^{\alpha^{\prime }},
\end{equation*}
where $C_{1}$ is a constant only depending on $\alpha$ and $\alpha^{\prime }$%
. Thus we obtain (\ref{e310}).
\end{proof}

\begin{proposition}
\label{pron2} Let $f_{i}$, $i=1,2$, satisfy (H1) and (H2'). Assume%
\begin{equation*}
Y_{t}^{i}=\xi^{i}+\int_{t}^{T}f_{i}(s,Y_{s}^{i},Z_{s}^{i})ds-%
\int_{t}^{T}Z_{s}^{i}dB_{s}-(K_{T}^{i}-K_{t}^{i}),
\end{equation*}
where $Y^{i}\in \mathbb{S}^{\alpha}(0,T)$, $Z^{i}\in \mathbb{H}%
^{\alpha}(0,T) $, $K^{i}$ is a decreasing process with $K_{0}^{i}=0$ and $%
K_{T}^{i}\in \mathbb{L}^{\alpha}(\Omega_{T})$ for some $\alpha>1$. Set $\hat{%
Y}_{t}=Y_{t}^{1}-Y_{t}^{2},\hat{Z}_{t}=Z_{t}^{1}-Z_{t}^{2}$ and $\hat{K}%
_{t}=K_{t}^{1}-K_{t}^{2}$. Then there exists a constant $C_{\alpha}:=C(%
\alpha ,T,\underline{\sigma},L^{w})>0$ such that%
\begin{equation}
\mathbb{\hat{E}}[(\int_{0}^{T}|\hat{Z}_{s}|^{2}ds)^{\frac{\alpha}{2}}]\leq
C_{\alpha}\{ \Vert \hat{Y}\Vert_{\mathbb{S}^{\alpha}}^{\alpha}+\Vert \hat {Y}%
\Vert_{\mathbb{S}^{\alpha}}^{\frac{\alpha}{2}}\sum_{i=1}^{2}[||Y^{i}||_{%
\mathbb{S}^{\alpha}}^{\frac{\alpha}{2}}+||\int_{0}^{T}f_{s}^{i,0}ds||_{%
\alpha,G}^{\frac{\alpha}{2}}]\},  \label{e316}
\end{equation}
where $f_{s}^{i,0}=|f_{i}(s,0,0)|+L^{w}\varepsilon$, $i=1,2$.
\end{proposition}

\begin{proof}
Applying It\^{o}'s formula to $|\hat{Y}_{t}|^{2}$, by similar analysis as
that in Proposition \ref{pron1}, we have%
\begin{equation*}
||\hat{Z}||_{\mathbb{H}^{\alpha}}^{\alpha}\leq C_{\alpha}\{ \Vert \hat{Y}%
\Vert_{\mathbb{S}^{\alpha}}^{\alpha}+\Vert \hat{Y}\Vert_{\mathbb{D}%
^{\alpha}}^{\frac{\alpha}{2}}[||K_{T}^{1}||_{\alpha,G}^{\frac{\alpha}{2}%
}+||K_{T}^{2}||_{\alpha,G}^{\frac{\alpha}{2}}+||\int_{0}^{T}\hat{f}%
_{s}ds||_{\alpha ,G}^{\frac{\alpha}{2}}]\},
\end{equation*}
where $\hat{f}%
_{s}=|f_{1}(s,Y_{s}^{2},Z_{s}^{2})-f_{2}(s,Y_{s}^{2},Z_{s}^{2})|+L^{w}%
\varepsilon$. By Proposition \ref{pron1}, we obtain%
\begin{align*}
& ||K_{T}^{1}||_{\alpha,G}^{\frac{\alpha}{2}}+||K_{T}^{2}||_{\alpha,G}^{%
\frac{\alpha}{2}}+||\int_{0}^{T}\hat{f}_{s}ds||_{\alpha,G}^{\frac{\alpha}{2}}
\\
& \leq C_{\alpha}\{||Y^{1}||_{\mathbb{S}^{\alpha}}^{\frac{\alpha}{2}%
}+||Y^{2}||_{\mathbb{S}^{\alpha}}^{\frac{\alpha}{2}}+||%
\int_{0}^{T}f_{s}^{1,0}ds||_{\alpha,G}^{\frac{\alpha}{2}}+||%
\int_{0}^{T}f_{s}^{2,0}ds||_{\alpha,G}^{\frac{\alpha}{2}}\}.
\end{align*}
Thus we get (\ref{e316}).
\end{proof}

\begin{proposition}
\label{pro3.5} Let $\xi ^{i}\in L_{G}^{\beta }(\Omega _{T})$ with $\beta >1$%
, $i=1,2$, and $f_{i}$ satisfy (H1) and (H2'). Assume that $%
(Y^{i},Z^{i},K^{i})\in \mathfrak{S}_{G}^{\alpha }(0,T)$ for some $1<\alpha
<\beta $ are the solutions of equation (\ref{e3}) corresponding to $\xi ^{i}$
and $f_{i}$ . Set $\hat{Y}_{t}=Y_{t}^{1}-Y_{t}^{2},\hat{Z}%
_{t}=Z_{t}^{1}-Z_{t}^{2}$ and $\hat{K}_{t}=K_{t}^{1}-K_{t}^{2}$. Then

\begin{description}
\item[(i)] There exists a constant $C_{\alpha}:=C(\alpha,T,\underline{\sigma
},L_{1}^{w})>0$ such that
\begin{equation}
|\hat{Y}_{t}|^{\alpha}\leq C_{\alpha}\mathbb{\hat{E}}_{t}[|\hat{\xi}%
|^{\alpha }+\int_{t}^{T}|\hat{f}_{s}|^{\alpha}ds],  \label{e317}
\end{equation}
where $\hat{f}%
_{s}=|f_{1}(s,Y_{s}^{2},Z_{s}^{2})-f_{2}(s,Y_{s}^{2},Z_{s}^{2})|+L_{1}^{w}%
\varepsilon$.

\item[(ii)] For any given $\alpha^{\prime }$ with $\alpha<\alpha^{\prime
}<\beta$, there exists a constant $C_{\alpha, \alpha^{\prime }}$ depending
on $\alpha $, $\alpha^{\prime }$, $T$, $\underline{\sigma}$, $L^{w}$ such
that%
\begin{align}
\mathbb{\hat{E}}[\sup_{t\in \lbrack0,T]}|\hat{Y}_{t}|^{\alpha}] & \leq
C_{\alpha, \alpha^{\prime }}\{ \mathbb{\hat{E}}[\sup_{t\in \lbrack0,T]}%
\mathbb{\hat{E}}_{t}[|\hat{\xi }|^{\alpha}]]  \notag \\
& +(\mathbb{\hat{E}}[\sup_{t\in \lbrack0,T]}\mathbb{\hat{E}}%
_{t}[(\int_{0}^{T}\hat{f}_{s}ds)^{\alpha^{\prime }}]])^{\frac{\alpha}{%
\alpha^{\prime }}}+\mathbb{\hat {E}}[\sup_{t\in \lbrack0,T]}\mathbb{\hat{E}}%
_{t}[(\int_{0}^{T}\hat{f}_{s}ds)^{\alpha^{\prime }}]]\}.  \label{e318}
\end{align}
\end{description}
\end{proposition}

\begin{proof}
For any $\gamma,\epsilon>0$, applying It\^{o}'s formula to $(|\hat{Y}%
_{t}|^{2}+\epsilon_{\alpha})^{\alpha/2}e^{\gamma t}$, where $%
\epsilon_{\alpha }=\epsilon(1-\alpha/2)^{+}$, by similar analysis as in
Proposition \ref{pro3.4}, we have by setting $\gamma=\alpha L^{w}+\alpha+%
\frac {\alpha(L^{w})^{2}}{\underline{\sigma}^{2}(\alpha-1)}$%
\begin{align*}
& (|\hat{Y}_{t}|^{2}+\epsilon_{\alpha})^{\alpha/2}e^{\gamma
t}+\int_{t}^{T}\alpha e^{\gamma s}(|\hat{Y}_{s}|^{2}+\epsilon_{\alpha})^{%
\alpha/2-1}\hat{Y}_{s}\hat{Z}_{s}dB_{s}+J_{T}-J_{t} \\
& \leq(|\hat{\xi}|^{2}+\epsilon_{\alpha})^{\alpha/2}e^{\gamma
T}+\int_{t}^{T}e^{\gamma s}|\hat{f}_{s}|^{\alpha}ds
\end{align*}
and
\begin{align*}
& (|\hat{Y}_{t}|^{2}+\epsilon_{\alpha})^{\alpha/2}e^{\gamma
t}+\int_{t}^{T}\alpha e^{\gamma s}(|\hat{Y}_{s}|^{2}+\epsilon_{\alpha})^{%
\alpha/2-1}\hat{Y}_{s}\hat{Z}_{s}dB_{s}+J_{T}-J_{t} \\
& \leq(|\hat{\xi}|^{2}+\epsilon_{\alpha})^{\alpha/2}e^{\gamma
T}+\int_{t}^{T}\alpha e^{\gamma s}(|\hat{Y}_{s}|^{2}+\epsilon_{\alpha})^{%
\alpha /2-1/2}\hat{f}_{s}ds,
\end{align*}
where%
\begin{equation*}
J_{t}=\int_{0}^{t}\alpha e^{\gamma s}(|\hat{Y}_{s}|^{2}+\epsilon_{\alpha
})^{\alpha/2-1}(\hat{Y}_{s}^{+}dK_{s}^{1}+\hat{Y}_{s}^{-}dK_{s}^{2}).
\end{equation*}
By Lemma \ref{lem3.3}, $J_{t}$ is a $G$-martingale. Taking conditional $G$%
-expectation and letting $\epsilon \downarrow0$, we obtain a constant $%
C_{\alpha}:=C_{\alpha}(T,L_{1}^{w},\underline{\sigma})>0$ such that%
\begin{equation*}
|\hat{Y}_{t}|^{\alpha}\leq C_{\alpha}\mathbb{\hat{E}}_{t}[|\hat{\xi}%
|^{\alpha }+\int_{t}^{T}|\hat{f}_{s}|^{\alpha}ds]
\end{equation*}
and
\begin{equation*}
|\hat{Y}_{t}|^{\alpha}\leq C_{\alpha}\mathbb{\hat{E}}_{t}[|\hat{\xi}%
|^{\alpha }+\int_{t}^{T}|\hat{Y}_{s}|^{\alpha-1}\hat{f}_{s}ds].
\end{equation*}
By the same analysis as that in Proposition \ref{pro3.4}, we get (\ref{e318}%
).
\end{proof}

\section{Existence and uniqueness of $G$-BSDEs}

In order to prove the existence of equation (\ref{e3}), we start with the
simple case $f(t,\omega ,y,z)=h(y,z)$, $\xi =\varphi (B_{T})$. Here $h\in
C_{0}^{\infty }(\mathbb{R}^{2})$, $\varphi \in C_{b.Lip}(\mathbb{R}^{2})$.
For this case, we can obtain the solution of equation (\ref{e3}) from the
following nonlinear partial differential equation:
\begin{equation}
\partial _{t}u+G(\partial _{xx}^{2}u)+h(u,\partial _{x}u)=0,u(T,x)=\varphi
(x).  \label{e4}
\end{equation}%
Then we approximate the solution of equation (\ref{e3}) with more
complicated $f$ by those of equations (\ref{e3}) with much simpler $\{f_{n}\}
$. More precisely, assume that $\Vert f_{n}-f\Vert _{M_{G}^{\beta
}}\rightarrow 0$ and $(Y^{n},Z^{n},K^{n})$ is the solution of the following $%
G$-BSDE
\begin{equation*}
Y_{t}^{n}=\xi
+\int_{t}^{T}f_{n}(s,Y_{s}^{n},Z_{s}^{n})ds-%
\int_{t}^{T}Z_{s}^{n}dB_{s}-(K_{T}^{n}-K_{t}^{n}).
\end{equation*}%
We try to prove that $(Y^{n},Z^{n},K^{n})$ converges to $(Y,Z,K)$ and $%
(Y,Z,K)$ is the solution of the following $G$-BSDE
\begin{equation*}
Y_{t}=\xi
+\int_{t}^{T}f(s,Y_{s},Z_{s})ds-\int_{t}^{T}Z_{s}dB_{s}-(K_{T}-K_{t}).
\end{equation*}

One of the main results of this paper is

\begin{theorem}
\label{the4.1} Assume that $\xi \in L_{G}^{\beta}(\Omega_{T})$ for some $%
\beta>1$ and $f$ satisfies (H1) and (H2). Then equation (\ref{e3}) has a
unique solution $(Y,Z,K)$. Moreover, for any $1<\alpha<\beta$ we have $Y\in
S^\alpha_G(0,T)$, $Z\in H_{G}^{\alpha}(0,T)$ and $K_{T}\in L_{G}^{\alpha
}(\Omega_{T})$.
\end{theorem}

\begin{proof}
The uniqueness of the solution is a direct consequence of the a
priori estimates in Proposition 3.8 and Proposition 3.9. By these
estimates it also suffices to prove the existence for the case $\xi
\in L_{ip}(\Omega_{T})$ and then pass to the limit for the general
situation.

Step 1. $f(t,\omega, y,z)= h(y,z)$ with $h\in C_{0}^{\infty}(\mathbb{R}^{2})$%
.

Part 1. We first consider the case $\xi=\varphi(B_{T}-B_{t_{1}})$ with $%
\varphi \in C_{b,Lip}(\mathbb{R})$ and $t_{1}<T$. Let $u$ be the solution of
equation (\ref{e4}) with terminal condition $\varphi$. By Theorem 6.4.3 in
Krylov \cite{Kr} (see also Theorem 4.4 in Appendix C in Peng \cite{P10}),
there exists a constant $\alpha \in(0,1)$ such that for each $\kappa>0$,%
\begin{equation*}
||u||_{C^{1+\alpha/2,2+\alpha}([0,T-\kappa]\times \mathbb{R})}<\infty.
\end{equation*}
Applying It\^{o}'s formula to $u(t,B_{t}-B_{t_{1}})$ on $[t_{1},T-\kappa]$,
we get%
\begin{align}
u(t,B_{t}-B_{t_{1}})= &
u(T-\kappa,B_{T-\kappa}-B_{t_{1}})+\int_{t}^{T-\kappa}h(u,%
\partial_{x}u)(s,B_{s}-B_{t_{1}})ds  \notag \\
& -\int_{t}^{T-\kappa}\partial_{x}u(s,B_{s}-B_{t_{1}})dB_{s}-(K_{T-\kappa
}-K_{t}),  \label{eab4}
\end{align}
where $K_{t}=\frac{1}{2}\int_{t_{1}}^{t}%
\partial_{xx}^{2}u(s,B_{s}-B_{t_{1}})d\langle
B\rangle_{s}-\int_{t_{1}}^{t}G(\partial_{xx}^{2}u(s,B_{s}-B_{t_{1}}))ds$ is
a nonincreasing $G$-martingale. We now prove that there exists a constant $%
L_{1}>0$ such that%
\begin{equation}
|u(t,x)-u(s,y)|\leq L_{1}(\sqrt{|t-s|}+|x-y|),\ t,s\in \lbrack0,T],x,y\in
\mathbb{R}.  \label{ea4}
\end{equation}
For each fixed $x_{0}\in \mathbb{R}$, set $\tilde{u}(t,x)=u(t,x+x_{0})$, it
is easy to check that $\tilde{u}$ is the solution of the following PDE:%
\begin{equation}
\partial_{t}\tilde{u}+G(\partial_{xx}^{2}\tilde{u})+h(\tilde{u},\partial _{x}%
\tilde{u})=0,\ \tilde{u}(T,x)=\varphi(x+x_{0}).  \label{ee4}
\end{equation}
Define $\hat{u}(t,x)=u(t,x)+L_{\varphi}|x_{0}|\exp(L_{h}(T-t))$, where $%
L_{\varphi}$ and $L_{h}$ are the Lipschitz constants of $\varphi$ and $h$
respectively, it is easy to verify that $\hat{u}$ is a supersolution of PDE (%
\ref{ee4}). Thus by comparison theorem (see Theorem 2.4 in Appendix C in
Peng \cite{P10}) we get%
\begin{equation*}
u(t,x+x_{0})\leq u(t,x)+L_{\varphi}|x_{0}|\exp(L_{h}(T-t)),\ t\in
\lbrack0,T],x\in \mathbb{R}.
\end{equation*}
Since $x_{0}$ is arbitrary, we get $|u(t,x)-u(t,y)|\leq \hat{L}|x-y|$, where
$\hat{L}=L_{\varphi}\exp(L_{h}T)$. From this we can get $|\partial
_{x}u(t,x)|\leq \hat{L}$ for each $t\in \lbrack0,T]$, $x\in \mathbb{R}$. On
the other hand, for each fixed $\bar{t}<\hat{t}<T$ and $x\in \mathbb{R}$,
applying It\^{o}'s formula to $u(s,x+B_{s}-B_{\bar{t}})$ on $[\bar{t},\hat{t}%
]$, we get%
\begin{equation*}
u(\bar{t},x)=\mathbb{\hat{E}}[u(\hat{t},x+B_{\hat{t}}-B_{\bar{t}})+\int _{%
\bar{t}}^{\hat{t}}h(u,\partial_{x}u)(s,x+B_{s}-B_{\bar{t}})ds].
\end{equation*}
From this we deduce that
\begin{equation*}
|u(\bar{t},x)-u(\hat{t},x)|\leq \mathbb{\hat{E}}[\hat{L}|B_{\hat{t}}-B_{%
\bar {t}}|+\tilde{L}|\hat{t}-\bar{t}|]\leq(\hat{L}\bar{\sigma}+\tilde{L}%
\sqrt {T})\sqrt{|\hat{t}-\bar{t}|},
\end{equation*}
where $\tilde{L}=\sup_{(x,y)\in \mathbb{R}^{2}}|h(x,y)|$. Thus we get (\ref%
{ea4}) by taking $L_{1}=\max \{ \hat{L},\hat{L}\bar{\sigma}+\tilde{L}\sqrt{T}%
\}$. Letting $\kappa \downarrow0$ in equation (\ref{eab4}), it is easy to
verify that
\begin{equation*}
\mathbb{\hat{E}}[|Y_{T-\kappa}-\xi|^{2}+\int_{T-%
\kappa}^{T}|Z_{t}|^{2}dt+(K_{T-\kappa}-K_{T})^{2}]\rightarrow0,
\end{equation*}
where $Y_{t}=u(t,B_{t}-B_{t_{1}})$ and $Z_{t}=%
\partial_{x}u(t,B_{t}-B_{t_{1}})$. Thus $(Y_{t},Z_{t},K_{t})_{t\in \lbrack
t_{1},T]}$ is a solution of equation (\ref{e3}) with terminal value $%
\xi=\varphi(B_{T}-B_{t_{1}})$. Furthermore, it is easy to check that $Y\in
S_{G}^{\alpha}(t_{1},T)$, $Z\in H_{G}^{\alpha}(t_{1},T)$ and $K_{T}\in
L_{G}^{\alpha}(\Omega_{T})$ for any $\alpha>1$.

Part 2. We now consider the case $\xi =\psi (B_{t_{1}},B_{T}-B_{t_{1}})$
with $\psi \in C_{b,Lip}(\mathbb{R}^{2})$, and the more general case can be
proved similarly. For each fixed $x\in \mathbb{R}$, let $u(\cdot ,x,\cdot )$
be the solution of equation (\ref{e4}) with terminal condition $\psi
(x,\cdot )$. By Part 1, we have%
\begin{align}
u(t,x,B_{t}-B_{t_{1}})=& u(T,x,B_{T}-B_{t_{1}})+\int_{t}^{T}h(u,\partial
_{y}u)(s,x,B_{s}-B_{t_{1}})ds  \notag \\
& -\int_{t}^{T}\partial
_{y}u(s,x,B_{s}-B_{t_{1}})dB_{s}-(K_{T}^{x}-K_{t}^{x}),  \label{par2}
\end{align}%
where $K_{t}^{x}=\frac{1}{2}\int_{t_{1}}^{t}\partial
_{yy}^{2}u(s,x,B_{s}-B_{t_{1}})d\langle B\rangle
_{s}-\int_{t_{1}}^{t}G(\partial _{yy}^{2}u(s,x,B_{s}-B_{t_{1}}))ds$. We
replace $x$ by $B_{t_{1}}$ and get%
\begin{equation*}
Y_{t}=Y_{T}+\int_{t}^{T}h(Y_{s},Z_{s})ds-%
\int_{t}^{T}Z_{s}dB_{s}-(K_{T}-K_{t}),
\end{equation*}%
where $Y_{t}=u(t,B_{t_{1}},B_{t}-B_{t_{1}})$, $Z_{t}=\partial
_{y}u(t,B_{t_{1}},B_{t}-B_{t_{1}})$ and%
\begin{equation*}
K_{t}=\frac{1}{2}\int_{t_{1}}^{t}\partial
_{yy}^{2}u(s,B_{t_{1}},B_{s}-B_{t_{1}})d\langle B\rangle
_{s}-\int_{t_{1}}^{t}G(\partial _{yy}^{2}u(s,B_{t_{1}},B_{s}-B_{t_{1}}))ds.
\end{equation*}%
Now we are in a position to prove $(Y,Z,K)\in \mathfrak{S}_{G}^{\alpha }(0,T)
$. We use the following argument, for each given $n\in \mathbb{N}$, by
partition of unity theorem, there exist $h_{i}^{n}\in C_{0}^{\infty }(%
\mathbb{R})$ with the diameter of support
$\lambda$(supp($h_{i}^{n}$))$<1/n$, $0\leq h_{i}^{n}\leq 1$,
$I_{[-n,n]}(x)\leq \sum_{i=1}^{k_{n}}h_{i}^{n}\leq 1$.
Choose $x_{i}^{n}$ such that $h_{i}^{n}(x_{i}^{n})>0$. Through equation (\ref%
{par2}), we have
\begin{equation*}
Y_{t}^{n}=Y_{T}^{n}+\int_{t}^{T}%
\sum_{i=1}^{n}h(y_{s}^{n,i},z_{s}^{n,i})h_{i}^{n}(B_{t_{1}})ds-%
\int_{t}^{T}Z_{s}^{n}dB_{s}-(K_{T}^{n}-K_{t}^{n}),
\end{equation*}%
where $y_{t}^{n,i}=u(t,x_{i}^{n},B_{t}-B_{t_{1}})$, $z_{t}^{n,i}=\partial
_{y}u(t,x_{i}^{n},B_{t}-B_{t_{1}})$, $Y_{t}^{n}=%
\sum_{i=1}^{n}y_{t}^{n,i}h_{i}^{n}(B_{t_{1}})$, $Z_{t}^{n}=%
\sum_{i=1}^{n}z_{t}^{n,i}h_{i}^{n}(B_{t_{1}})$ and $K_{t}^{n}=%
\sum_{i=1}^{n}K_{t}^{x_{i}^{n}}h_{i}^{n}(B_{t_{1}}).$

By the same analysis as that in Part 1, we can obtain a constant $L_{2}>0$
such that for each $t$, $s\in \lbrack0,T]$, $x$, $x^{\prime}$, $y$, $%
y^{\prime}\in \mathbb{R}$,
\begin{equation*}
|u(t,x,y)-u(s,x^{\prime},y^{\prime})|\leq L_{2}(\sqrt{|t-s|}+|x-x^{\prime
}|+|y-y^{\prime}|).
\end{equation*}
From this we get
\begin{align*}
|Y_{t}-Y_{t}^{n}| & \leq
\sum_{i=1}^{k_{n}}h_{i}^{n}(B_{t_{1}})|u(t,x_{i}^{n},B_{t}-B_{t_{1}})-u(t,B_{t_{1}},B_{t}-B_{t_{1}})|+|Y_{t}|I_{[|B_{t_{1}}|>n]}
\\
& \leq \frac{L_{2}}{n}+\frac{||u||_{\infty}}{n}|B_{t_{1}}|.
\end{align*}
Thus
\begin{equation*}
\hat{\mathbb{E}}[\sup_{t\in \lbrack t_{1},T]}|Y_{t}-Y_{t}^{n}|^{\alpha}]\leq
\hat{\mathbb{E}}[(\frac{L_{2}}{n}+\frac{||u||_{\infty}}{n}%
|B_{t_{1}}|)^{\alpha}]\rightarrow0.
\end{equation*}
By Proposition \ref{pron2}, we have
\begin{equation*}
\hat{\mathbb{E}}[(\int_{t_{1}}^{T}|Z_{s}-Z_{s}^{n}|^{2}ds)^{\alpha/2}]\leq
C_{\alpha}\{ \hat{\mathbb{E}}[\sup_{t\in \lbrack
t_{1},T]}|Y_{t}-Y_{t}^{n}|^{\alpha}]+(\hat{\mathbb{E}}[\sup_{t\in \lbrack
t_{1},T]}|Y_{t}-Y_{t}^{n}|^{\alpha}])^{1/2}\},
\end{equation*}
where $C_{\alpha}>0$ is a constant depending only on $\alpha$, $T$, $L^{w}$
and $\underline{\sigma}$, thus we obtain $\hat{\mathbb{E}}%
(\int_{t_{1}}^{T}|Z_{s}-Z_{s}^{n}|^{2}ds)^{\alpha/2}\rightarrow0$, which
implies that $Z\in H_{G}^{\alpha}(t_{1},T)$ for any $\alpha>1$. By $%
K_{t}=Y_{t}-Y_{t_{1}}+\int_{t_{1}}^{t}h(Y_{s},Z_{s})ds-%
\int_{t_{1}}^{t}Z_{s}dB_{s}$, we obtain $K_{t}\in L_{G}^{\alpha}(\Omega_{t})$
for any $\alpha>1$. We now proceed to prove that $K$ is a $G$-martingale.
Following the framework in Li and Peng \cite{L-P}, we take
\begin{equation*}
h_{i}^{n}(x)=I_{[-n+\frac{i}{n},-n+\frac{i+1}{n})}(x),i=0,\ldots,2n^{2}-1,
\end{equation*}
$h_{2n^{2}}^{n}=1-\sum_{i=0}^{2n^{2}-1}h_{i}^{n}$. Through equation (\ref%
{par2}), we get
\begin{equation*}
\tilde{Y}_{t}^{n}=\tilde{Y}_{T}^{n}+\int_{t}^{T}h(\tilde{Y}_{s}^{n},\tilde {Z%
}_{s}^{n})ds-\int_{t}^{T}\tilde{Z}_{s}^{n}dB_{s}-(\tilde{K}_{T}^{n}-\tilde{K}%
_{t}^{n}),
\end{equation*}
where $\tilde{Y}_{t}^{n}=\sum_{i=0}^{2n^{2}}u(t,-n+\frac{i}{n}%
,B_{t}-B_{t_{1}})h_{i}^{n}(B_{t_{1}})$, $\tilde{Z}_{t}^{n}=%
\sum_{i=0}^{2n^{2}}\partial _{y}u(t,-n+\frac{i}{n}%
,B_{t}-B_{t_{1}})h_{i}^{n}(B_{t_{1}})$ and $\tilde {K}_{t}^{n}=%
\sum_{i=0}^{2n^{2}}K_{t}^{-n+\frac{i}{n}}h_{i}^{n}(B_{t_{1}})$. By
Proposition \ref{pron2}, we have $\hat{\mathbb{E}}[(\int_{t_{1}}^{T}|Z_{s}-%
\tilde{Z}_{s}^{n}|^{2}ds)^{\alpha/2}]\rightarrow0$ for any $\alpha>1$. Thus
we get $\hat{\mathbb{E}}[|K_{t}-\tilde{K}_{t}^{n}|^{\alpha}]\rightarrow0$
for any $\alpha>1$. By Proposition \ref{proA.8}, we obtain for each $%
t_{1}\leq t<s\leq T$,%
\begin{align*}
\hat{\mathbb{E}}[|\hat{\mathbb{E}}_{t}[K_{s}]-K_{t}|] & =\hat{\mathbb{E}}[|%
\mathbb{\hat{E}}_{t}[K_{s}]-\mathbb{\hat{E}}_{t}[\tilde{K}_{s}^{n}]+\tilde{K}%
_{t}^{n}-K_{t}|] \\
& \leq \hat{\mathbb{E}}[\mathbb{\hat{E}}_{t}[|K_{s}-\tilde{K}_{s}^{n}|]]+%
\hat{\mathbb{E}}[|\tilde{K}_{t}^{n}-K_{t}|] \\
& =\hat{\mathbb{E}}[|K_{s}-\tilde{K}_{s}^{n}|]+\hat{\mathbb{E}}[|\tilde {K}%
_{t}^{n}-K_{t}|]\rightarrow0.
\end{align*}
Thus we get $\mathbb{\hat{E}}_{t}[K_{s}]=K_{t}$. For $%
Y_{t_{1}}=u(t_{1},B_{t_{1}},0)$, we can use the same method as Part 1 on $%
[0,t_{1}]$.

Step 2. $f(t,\omega, y,z)= \sum_{i=1}^{N} f^{i}h^{i}(y,z)$ with $f^{i}\in
M^{0}_{G}(0,T)$ and $h^{i}\in C_{0}^{\infty}(\mathbb{R}^{2})$.

The analysis is similar to Part 2 of Step 1.

Step 3. $f(t,\omega, y,z)= \sum_{i=1}^{N} f^{i}h^{i}(y,z)$ with $f^{i}\in
M^{\beta}_{G}(0,T)$ bounded and $h^{i}\in C_{0}^{\infty}(\mathbb{R}^{2})$, $%
h^{i}\geq0$ and $\sum_{i=1}^{N}h^{i}\leq1$.

Choose $f^{i}_{n}\in M^{0}_{G}(0,T)$ such that $|f^{i}_{n}|\leq
\|f^{i}\|_{\infty}$ and $\sum_{i=1}^{N}
\|f^{i}_{n}-f^{i}\|_{M^{\beta}_{G}}<1/n$. Set $f_{n}=\sum_{i=1}^{N}
f^{i}_{n}h^{i}(y,z)$, which are uniformly Lipschitz. Let $(Y^{n},Z^{n},
K^{n})$ be the solution of equation (\ref{e3}) with generator $f_{n}$.

Noting that
\begin{equation*}
\hat{f}%
_{s}^{m,n}:=|f_{m}(s,Y_{s}^{n},Z_{s}^{n})-f_{n}(s,Y_{s}^{n},Z_{s}^{n})|\leq
\sum_{i=1}^{N}|f_{n}^{i}-f^{i}|+\sum_{i=1}^{N}|f_{m}^{i}-f^{i}|=:\hat{f}_{n}+%
\hat{f}_{m},
\end{equation*}
we have, for any $1<\alpha<\beta$,
\begin{equation*}
\mathbb{\hat{E}}_{t}[(\int_{0}^{T}\hat{f}_{s}^{m,n}ds)^{\alpha}]\leq \mathbb{%
\hat{E}}_{t}[(\int_{0}^{T}(|\hat{f}_{n}(s)|+|\hat{f}_{m}(s)|)ds)^{\alpha}].
\end{equation*}
Thus by Theorem \ref{the2.10}, we get $||\int_{0}^{T}\hat{f}%
_{s}^{m,n}ds||_{\alpha,\mathcal{E}}\rightarrow0$ as $m,n\rightarrow \infty$
for any $\alpha \in(1,\beta)$. By Proposition \ref{pro3.5} we know that $%
\{Y^{n}\}$ is a cauchy sequence under the norm $\Vert \cdot
\Vert_{S_{G}^{\alpha}}$. By Proposition \ref{pro3.4} and Proposition \ref%
{pron2}, $\{Z^{n}\}$ is a cauchy sequence under the norm $\Vert \cdot
\Vert_{H_{G}^{\alpha}}$. In order to show that $\{K_{T}^{n}\}$ is a cauchy
sequence under the norm $\Vert \cdot \Vert_{L_{G}^{\alpha}}$, it suffices to
prove $\{ \int_{0}^{T}f_{n}(s,Y_{s}^{n},Z_{s}^{n})ds\}$ is a cauchy sequence
under the norm $\Vert \cdot \Vert_{L_{G}^{\alpha}}$. In fact,
\begin{align*}
& |f_{n}(s,Y^{n},Z^{n})-f_{m}(s,Y^{m},Z^{m})| \\
&
\leq|f_{m}(s,Y^{n},Z^{n})-f_{m}(s,Y^{m},Z^{m})|+|f_{n}(s,Y^{n},Z^{n})-f_{m}(s,Y^{n},Z^{n})|
\\
& \leq L(|\hat{Y}_{s}|+|\hat{Z}_{s}|)+\hat{f}_{n}+\hat{f}_{m},
\end{align*}
which implies the desired result.

Step 4. $f$ is bounded, Lipschitz. $|f(t,\omega, y,z)|\leq CI_{B(R)}(y,z)$
for some $C, R>0$. Here $B(R)=\{(y,z)| y^{2}+z^{2}\leq R^{2}\}$.

For any $n$, by the partition of unity theorem, there exists $%
\{h_{n}^{i}\}_{i=1}^{N_{n}}$ such that $h_{n}^{i}\in C_{0}^{\infty}(\mathbb{R%
}^{2})$, the diameter  of support
$\lambda$(supp($h_{n}^{i}$))$<1/n$, $0\leq
h_{n}^{i}\leq1$, $I_{B(R)}\leq \sum_{i=1}^{N}h_{n}^{i}\leq1$. Then $%
f(t,\omega,y,z)=\sum _{i=1}^{N}f(t,\omega,y,z)h_{n}^{i}$. Choose $%
y_{n}^{i},z_{n}^{i}$ such that $h_{n}^{i}(y_{n}^{i},z_{n}^{i})>0$. Set $%
f_{n}(t,\omega,y,z)=\sum_{i=1}^{N}f(t,\omega,y_{n}^{i},z_{n}^{i})h_{n}^{i}$.
Then
\begin{equation*}
|f(t,\omega,y,z)-f_{n}(t,\omega,y,z)|\leq \sum_{i=1}^{N}|f(t,\omega
,y,z)-f(t,\omega,y_{n}^{i},z_{n}^{i})|h_{n}^{i}\leq L/n
\end{equation*}
and
\begin{equation*}
|f_{n}(t,\omega,y,z)-f_{n}(t,\omega,y^{\prime},z^{\prime})|\leq
L(|y-y^{\prime }|+|z-z^{\prime}|+2/n).
\end{equation*}
Noting that $|f_{m}(s,Y_{s}^{n},Z_{s}^{n})-f_{n}(s,Y_{s}^{n},Z_{s}^{n})|%
\leq(L/n+L/m)$, we have
\begin{equation*}
\mathbb{\hat{E}}_{t}[|%
\int_{0}^{T}(|f_{m}(s,Y_{s}^{n},Z_{s}^{n})-f_{n}(s,Y_{s}^{n},Z_{s}^{n})|+%
\frac{2L}{m})ds|^{\alpha}]\leq T^{\alpha}(\frac{L}{n}+\frac{3L}{m})^{\alpha}.
\end{equation*}
So by Proposition \ref{pro3.5} we conclude that $\{Y^{n}\}$ is a cauchy
sequence under the norm $\Vert \cdot \Vert_{S_{G}^{\alpha}}$. Consequently, $%
\{Z^{n}\}$ is a cauchy sequence under the norm $\Vert \cdot \Vert
_{H_{G}^{\alpha}}$ by Proposition \ref{pro3.4} and Proposition \ref{pron2}.
Now we shall prove $\{ \int_{0}^{T}f_{n}(s,Y_{s}^{n},Z_{s}^{n})ds\}$ is a
cauchy sequence under the norm $\Vert \cdot \Vert_{L_{G}^{\alpha}}$. In
fact,
\begin{align*}
& |f_{n}(s,Y^{n},Z^{n})-f_{m}(s,Y^{m},Z^{m})| \\
&
\leq|f_{m}(s,Y^{n},Z^{n})-f_{m}(s,Y^{m},Z^{m})|+|f_{n}(s,Y^{n},Z^{n})-f_{m}(s,Y^{n},Z^{n})|
\\
& \leq L(|\hat{Y}_{s}|+|\hat{Z}_{s}|+2/m)+L/n+L/m,
\end{align*}
which implies the desired result.

Step 5. $f$ is bounded, Lipschitz.

For any $n\in \mathbb{N}$, choose $h^{n}\in C_{0}^{\infty}(\mathbb{R}^{2})$
such that $I_{B(n)}\leq h^{n}\leq I_{B(n+1)}$ and $\{h^{n}\}$ are uniformly
Lipschitz w.r.t. $n$. Set $f_{n}=fh^{n}$, which are uniformly Lipschitz.
Noting that for $m>n$
\begin{align*}
& |f_{m}(s,Y_{s}^{n},Z_{s}^{n})-f_{n}(s,Y_{s}^{n},Z_{s}^{n})| \\
& \leq|f(s,Y_{s}^{n},Z_{s}^{n})|I_{[|Y_{s}^{n}|^{2}+|Z_{s}^{n}|^{2}>n^{2}]}
\\
& \leq \Vert f\Vert_{\infty}\frac{|Y_{s}^{n}|+|Z_{s}^{n}|}{n},
\end{align*}
we have
\begin{align*}
& \mathbb{\hat{E}}_{t}[(%
\int_{0}^{T}|f_{m}(s,Y_{s}^{n},Z_{s}^{n})-f_{n}(s,Y_{s}^{n},Z_{s}^{n})|ds)^{%
\alpha}] \\
& \leq \frac{\Vert f\Vert_{\infty}^{\alpha}}{n^{\alpha}}\mathbb{\hat{E}}%
_{t}[(\int_{0}^{T}|Y_{s}^{n}|+|Z_{s}^{n}|ds)^{\alpha}] \\
& \leq \frac{\Vert f\Vert_{\infty}^{\alpha}}{n^{\alpha}}C(\alpha,T)\mathbb{%
\hat {E}}_{t}[\int_{0}^{T}|Y_{s}^{n}|^{\alpha}ds+(%
\int_{0}^{T}|Z_{s}^{n}|^{2}ds)^{\alpha/2}],
\end{align*}
where $C(\alpha,T):=2^{\alpha-1}(T^{\alpha-1}+T^{\alpha/2}])$.

So by Theorem \ref{the2.10} and Proposition \ref{pro3.4} we get $||\int
_{0}^{T}\hat{f}_{s}^{m,n}ds||_{\alpha,\mathcal{E}}\rightarrow0$ as $%
m,n\rightarrow \infty$ for any $\alpha \in(1,\beta)$. By Proposition \ref%
{pro3.5}, we conclude that $\{Y^{n}\}$ is a cauchy sequence under the norm $%
\Vert \cdot \Vert_{S_{G}^{\alpha}}$. Consequently, $\{Z^{n}\}$ is a cauchy
sequence under the norm $\Vert \cdot \Vert_{H_{G}^{\alpha}}$. Now it
suffices to prove $\{ \int_{0}^{T}f_{n}(s,Y_{s}^{n},Z_{s}^{n})ds\}$ is a
cauchy sequence under the norm $\Vert \cdot \Vert_{L_{G}^{\alpha}}$. In
fact,
\begin{align*}
& |f_{n}(s,Y^{n},Z^{n})-f_{m}(s,Y^{m},Z^{m})| \\
&
\leq|f_{m}(s,Y^{n},Z^{n})-f_{m}(s,Y^{m},Z^{m})|+|f_{n}(s,Y^{n},Z^{n})-f_{m}(s,Y^{n},Z^{n})|
\\
& \leq L(|\hat{Y}_{s}|+|\hat{Z}%
_{s}|)+|f(s,Y_{s}^{n},Z_{s}^{n})|1_{[|Y_{s}^{n}|+|Z_{s}^{n}|>n]},
\end{align*}
which implies the desired result by Proposition \ref{pro3.4}.

Step 6. For the general $f$.

Set $f_{n}=[f\vee(-n)]\wedge n$, which are uniformly Lipschitz. Choose $%
0<\delta<\frac{\beta-\alpha}{\alpha}\wedge1$. Then $\alpha<\alpha^{\prime
}=(1+\delta)\alpha<\beta$. Since for $m>n$
\begin{equation*}
|f_{n}(s,Y^{n},Z^{n})-f_{m}(s,Y^{n},Z^{n})|%
\leq|f(s,Y_{s}^{n},Z_{s}^{n})|I_{[|f(s,Y_{s}^{n},Y_{s}^{n})|>n]}\leq \frac{1%
}{n^{\delta}}|f(s,Y_{s}^{n},Z_{s}^{n})|^{1+\delta},
\end{equation*}
we have
\begin{align*}
& \mathbb{\hat{E}}_{t}[(%
\int_{0}^{T}|f_{n}(s,Y^{n},Z^{n})-f_{m}(s,Y^{n},Z^{n})|ds)^{\alpha}] \\
& \leq \frac{1}{n^{\alpha \delta}}\mathbb{\hat{E}}_{t}[(%
\int_{0}^{T}|f(s,Y_{s}^{n},Z_{s}^{n})|^{1+\delta}ds)^{\alpha}], \\
& \leq \frac{C(\alpha,T,L,\delta)}{n^{\alpha \delta}}\mathbb{\hat{E}}%
_{t}[\int_{0}^{T}|f(s,0,0)|^{\alpha^{\prime}}ds+\int_{0}^{T}|Y_{s}^{n}|^{%
\alpha^{\prime}}ds+(\int_{0}^{T}|Z_{s}^{n}|^{2}ds)^{\frac{\alpha^{\prime}}{2}%
}],
\end{align*}
where $C(\alpha,T,L,\delta):=3^{\alpha^{\prime}-1}(T^{\alpha-1}+L^{\alpha
^{\prime}}T^{\frac{\alpha(1-\delta)}{2}}+T^{\alpha-1}L^{\alpha^{\prime}})$.
So by Theorem \ref{the2.10} and Proposition \ref{pro3.4} we get $%
||\int_{0}^{T}\hat{f}_{s}^{m,n}ds||_{\alpha,\mathcal{E}}\rightarrow0$ as $%
m,n\rightarrow \infty$ for any $\alpha \in(1,\beta)$. By Proposition \ref%
{pro3.5}, we know that $\{Y^{n}\}$ is a cauchy sequence under the norm $%
\Vert \cdot \Vert _{S_{G}^{\alpha}}$. And consequently $\{Z^{n}\}$ is a
cauchy sequence under the norm $\Vert \cdot \Vert_{H_{G}^{\alpha}}$. Now we
prove $\{ \int_{0}^{T}f_{n}(s,Y_{s}^{n},Z_{s}^{n})ds\}$ is a cauchy sequence
under the norm $\Vert \cdot \Vert_{L_{G}^{\alpha}}$. In fact,
\begin{align*}
& |f_{n}(s,Y^{n},Z^{n})-f_{m}(s,Y^{m},Z^{m})| \\
&
\leq|f_{m}(s,Y^{n},Z^{n})-f_{m}(s,Y^{m},Z^{m})|+|f_{n}(s,Y^{n},Z^{n})-f_{m}(s,Y^{n},Z^{n})|
\\
& \leq L(|\hat{Y}_{s}|+|\hat{Z}_{s}|)+\frac{3^{\delta}}{n^{\delta}}%
(|f_{s}^{0}|^{1+\delta}+|Y_{s}^{n}|^{1+\delta}+|Z_{s}^{n}|^{1+\delta}),
\end{align*}
which implies the desired result by Proposition \ref{pro3.4}.
\end{proof}

Moreover, we have the following result.

\begin{theorem}
\label{the4.4} Assume that $\xi \in L_{G}^{\beta}(\Omega_{T})$ for some $%
\beta>1$ and $f$, $g$ satisfy (H1) and (H2). Then equation (\ref{e1}) has a
unique solution $(Y,Z,K)$. Moreover, for any $1<\alpha<\beta$ we have $Y\in
S^\alpha_G(0,T)$, $Z\in H_{G}^{\alpha}(0,T)$ and $K_{T}\in L_{G}^{\alpha
}(\Omega_{T})$.
\end{theorem}

\begin{proof}
The proof is similar to that of Theorem \ref{the4.1}.
\end{proof}

\begin{remark}
\label{rem4.5} The above results still hold for the case $d>1$.
\end{remark}


\renewcommand{\refname}{\large References}{\normalsize \ }


\begin{thebibliography}{99}

\bibitem{Avel1995} Avellaneda, M., Levy, A. and Paras A. (1995). Pricing and
hedging derivative securities in markets with uncertain volatilities. Appl.
Math. Finance 2, 73-88.

\bibitem{Bismut} Bismut, J.M. (1973) Conjugate Convex Functions in Optimal
Stochastic Control, J.Math. Anal. Apl. 44, 384--404.

\bibitem{CHMP} Coquet, F., Hu, Y., Memin J. and Peng, S. (2002) Filtration
Consistent Nonlinear Expectations and Related g-Expectations, Probab. Theory
Relat. Fields 123, 1-27.

\bibitem{DenisMartini2006} Denis, L. and Martini, C. (2006) A Theoretical
Framework for the Pricing of Contingent Claims in the Presence of Model
Uncertainty, The Annals of Applied Probability, vol. 16, No. 2, pp 827-852.

\bibitem{DHP11} Denis, L., Hu, M. and Peng S.(2011) \emph{Function spaces
and capacity related to a sublinear expectation: application to $G$-Brownian
motion pathes,} Potential Anal., 34: 139-161.


\bibitem{EPQ} El Karoui, N., Peng, S., Quenez, M.C., Backward stochastic
differential equations in finance, Math. Finance 7, 1-71, 1997.

\bibitem{HP09} Hu, M. and Peng, S.(2009) \emph{On representation theorem of
G-expectations and paths of $G$-Brownian motion}. Acta Math. Appl. Sin.
Engl. Ser., 25,(3): 539-546, 2009.

\bibitem{Kr} Krylov, N.V.(1987) \emph{Nonlinear Parabolic and Elliptic
Equations of the Second Order,} Reidel Publishing Company. (Original Russian
Version by Nauka, Moscow, 1985).

\bibitem{L-P} Li, X and Peng, S.(2011) \emph{Stopping times and related
It\^o's calculus with $G$-Brownian motion,} Stochastic Processes and their
Applications, 121: 1492-1508.

\bibitem{PP90} Pardoux E. and Peng, S.(1990) \emph{Adapted Solutions of
Backward Stochastic Equations,} Systerm and Control Letters, 14: 55-61.

\bibitem{Peng1991} Peng, S. (1991) \emph{Probabilistic Interpretation for
Systems of Quasilinear Parabolic Partial Differential Equations, }%
Stochastics, 37, 61--74.

\bibitem{PP92} Pardoux, E. and Peng, S. (1992) \emph{Backward stochastic
differential equations and quasilinear parabolic partial differential
equations,} Stochastic partial differential equations and their
applications, Proc. IFIP, LNCIS 176, 200--217.

\bibitem{Peng1992} Peng, S. (1992) \emph{A Generalized Dynamic Programming
Principle and Hamilton-Jacobi-Bellmen equation,} Stochastics, 38, 119--134.

\bibitem{Peng1997} Peng, S. (1997) BSDE and related g-expectation, in Pitman
Research Notes in Mathematics Series, No. 364, Backward Stochastic
Differential Equation, N. El Karoui and L. Mazliak (edit.), 141-159.

\bibitem{Peng2004} Peng, S. (2004) \emph{Filtration consistent nonlinear
expectations and evaluations of contingent claims,} Acta Mathematicae
Applicatae Sinica, 20(2) 1--24.

\bibitem{Peng2005} Peng, S. (2005) \emph{Nonlinear expectations and
nonlinear Markov chains,} Chin. Ann. Math. 26B(2) 159--184.

\bibitem{P07a} Peng, S.(2007) \emph{$G$-expectation, $G$-Brownian Motion and
Related Stochastic Calculus of It\^o type}, Stochastic analysis and
applications, 541-567, Abel Symp., 2, Springer, Berlin.

\bibitem{P07b} Peng, S.(2007) \emph{$G$-Brownian Motion and Dynamic Risk
Measure under Volatility Uncertainty}, arXiv:0711.2834v1 [math.PR].

\bibitem{P08a} Peng, S.(2008) \emph{Multi-Dimensional $G$-Brownian Motion
and Related Stochastic Calculus under $G$-Expectation}, Stochastic Processes
and their Applications, 118(12): 2223-2253.

\bibitem{P08b} Peng, S.(2008) \emph{A New Central Limit Theorem under
Sublinear Expectations}, arXiv:0803.2656v1 [math.PR].

\bibitem{P09} Peng, S.(2009) \emph{Survey on normal distributions, central
limit theorem, Brownian motion and the related stochastic calculus under
sublinear expectations}, Science in China Series A: Mathematics, 52(7):
1391-1411.

\bibitem{P10} Peng, S.(2010) \emph{Nonlinear Expectations and Stochastic
Calculus under Uncertainty}, arXiv:1002.4546v1 [math.PR].

\bibitem{PengICM2010} Peng, S.(2010) \emph{Backward Stochastic Differential
Equation, Nonlinear Expectation and Their Applications}, in Proceedings of
the International Congress of Mathematicians Hyderabad, India, 2010.

\bibitem{PSZ2012} Peng, S., Song, Y. and Zhang, J. (2012) \emph{A Complete
Representation Theorem for G-martingales, }Preprint,\emph{\ }%
arXiv:1201.2629v1.

\bibitem{STZ} Soner, M., Touzi, N. and Zhang, J.(2011) \emph{Martingale
Representation Theorem under G-expectation,} Stochastic Processes and their
Applications, 121: 265-287.

\bibitem{STZ11} Soner M, Touzi N, Zhang J.(2012) \emph{Wellposedness of
Second Order Backward SDEs,} Probability Theory and Related Fields,
153(1-2): 149-190.

\bibitem{Song11} Song, Y.(2011) \emph{Some properties on G-evaluation and
its applications to G-martingale decomposition,} Science China Mathematics,
54(2): 287-300.

\bibitem{Song12} Song, Y.(2012) \emph{Uniqueness of the representation for $%
G $-martingales with finite variation,} Electron. J. Probab. 17 no. 24 1-15.


\end{thebibliography}
\end{document}